\documentclass[11pt,letterpaper]{amsart}
\usepackage{graphicx} 
\usepackage{amsthm}
\usepackage{amsmath}
\usepackage{amssymb}
\usepackage{comment}
\usepackage{tikz}
\usepackage{enumerate}
\usepackage{enumitem}
\usepackage{amsfonts}
\usepackage{leftindex}
\usepackage{mathtools}
\usepackage{hyperref}
\usepackage{cleveref}
\usepackage{todonotes}
\usepackage{booktabs}
\usepackage[utf8]{inputenc}
\usepackage[left=1.25in, right=1.25in, bottom=1.25in, top=1.25in]{geometry}

\newtheorem{theorem}{Theorem}[section]
\newtheorem{def-prop}[theorem]{Definition-Proposition}
\newtheorem{prop}[theorem]{Proposition}

\newtheorem{lemma}[theorem]{Lemma}

\newtheorem{cor}[theorem]{Corollary}

\theoremstyle{definition}
\newtheorem{ex}[theorem]{Example}

\newtheorem{defin}[theorem]{Definition}

\theoremstyle{remark}
\newtheorem{remark}{Remark}

\Crefname{remark}{Remark}{Remarks}
\Crefname{prop}{Proposition}{Propositions}

\DeclareTextFontCommand{\emph}{\color{blue}\em}

\newcommand{\R}{\mathbb{R}}
\newcommand{\Z}{\mathbb{Z}}

\begin{document}
\title{Coxeter Condorcet domains}
\author{Yibo Gao}
\address{Beijing International Center for Mathematical Research, Peking University, Beijing, China}
\email{gaoyibo@bicmr.pku.edu.cn}
\author{Yulin Peng}
\address{Qiuzhen College, Tsinghua University, Beijing, China}
\email{pyl24@mails.tsinghua.edu.cn}

\date{\today}

\begin{abstract}
Condorcet domains are subsets of permutations that ensure pairwise majority voting yields acyclic outcomes, and they form an active area of research at the intersection of social choice theory and combinatorics. In this paper, we extend the theory of Condorcet domains to the broader setting of arbitrary finite Coxeter groups. The core contribution of our approach is the introduction of Condorcet root posets, defined on the chosen root systems. Notably, we establish a natural bijection between closed Condorcet domains and Condorcet root posets, which facilitates the study of Condorcet domains. Using this correspondence, we extend the median graph representation of closed Condorcet domains to arbitrary finite Coxeter groups, demonstrating that these domains can be characterized by the skeletons of their associated Condorcet root posets. These results are novel even in type $A$. Furthermore, these posets give a unified language that efficiently captures a wide range of desirable properties of Condorcet domains, such as being maximal, connected, peak-pit, and of tiling type. Using this framework, we strengthen and generalize several classical results: we establish that a maximal Condorcet domain is connected if and only if it is peak-pit; we prove that the tiling-type property is equivalent to the combination of being maximal and connected, and having maximal width; and we show that strictly positive voting profiles on connected Condorcet domains yield outcomes with only simple ties.
\end{abstract}
\maketitle

\section{Introduction}\label{sec: intro}
A central problem in social choice theory is how to aggregate individual preferences into a coherent collective decision \cite{Arrow1951}. This challenge is already apparent under majority rule: as demonstrated by the Condorcet paradox \cite{Condorcet1785}, pairwise majority comparisons can yield cyclic outcomes. For instance, consider a case with three candidates and three voters whose respective preference rankings are $A \succ B \succ C$, $B \succ C \succ A$, and $C \succ A \succ B$. Pairwise majority voting then produces a cycle: $A$ defeats $B$, $B$ defeats $C$, yet $C$ defeats $A$. To avoid this scenario, the notion of a Condorcet domain was introduced \cite{Black1948, Ward1961, Sen1966}. By identifying the strict linear orders on $n$ candidates with permutations in the symmetric group $S_n$, a domain $\mathcal{D} \subseteq S_n$ is said to be \emph{Condorcet} if a profile of voters whose preferences are restricted to $\mathcal{D}$ can never generate a majority cycle. Since their introduction, Condorcet domains have been investigated both as fundamental tools for social choice and as rich combinatorial objects in their own right. A major line of subsequent research has focused on Condorcet domains that are maximal with respect to set inclusion. This focus has yielded various structural characterizations, revealing deep connections with weak Bruhat orders \cite{GalambosReiner2008}, rhombus tilings \cite{danilov2012condorcet}, and median graphs \cite{Puppe-Slinko-median-graph}.

In this paper, we generalize the concept of a Condorcet domain from the symmetric group to arbitrary finite Coxeter groups in a natural way. By introducing Condorcet root posets, we characterize closed Condorcet domains and their associated median graphs. These results are novel even in type $A$. Moreover, we use these posets to study several well-known classes of Condorcet domains, including maximal, connected, peak-pit, and tiling-type domains. And then we generalize some of their fundamental properties.

Let $\Phi$ be a finite root system, $\Phi^+$ be a choice of positive roots, $\Phi^-=\Phi\setminus\Phi^+$ be the negative roots, and $W=W(\Phi)$ be its corresponding Coxeter group. For each $w\in W$, its \emph{(left) inversion set} is $I(w):=\{\alpha\in\Phi\mid w^{-1}\alpha\in\Phi^-\}$. Note that we differ significantly from the standard literature by allowing $I(w)$ to contain negative roots, so that $|I(w)|=|\Phi|/2$ as exactly one of $\alpha,-\alpha$ lies in $I(w)$. We choose this convention to discuss all roots simultaneously, rather than separating $\Phi^+$ from $\Phi^-$. We say that $\{\alpha,\beta,\gamma\}\subseteq \Phi$ is a \emph{root-triple} if there exists $a,b,c\in\R_{>0}$ such that $a\alpha+b\beta+c\gamma=0$. For a set $A\subseteq \Phi$, write $-A:=\{-\alpha\mid \alpha \in A\}$. Then $\{\alpha,\beta,\gamma\}$ is a root-triple implies that $-\{\alpha,\beta,\gamma\}$ also is a root-triple.

The following is our key definition.
\begin{defin}\label{def: condorcet}
A nonempty subset $D\subseteq W$ is a \emph{(Coxeter) Condorcet domain} if there do not exist distinct $u,v,w\in D$ and a root-triple $\{\alpha,\beta,\gamma\}$ such that $\alpha,\beta\in I(u)$, $\alpha,\gamma\in I(v)$, $\beta,\gamma\in I(w)$.
\end{defin}
Given a Condorcet domain $D$, a \emph{voting profile} is a map $f:D\rightarrow \Z_{\geq0}$ that is not constantly zero. For each $\alpha\in\Phi$, write $f(\alpha)=\sum_{w\in D,\ \alpha\in I(w)}f(w)$. Write $|f|:=\sum_{w\in D}f(w)$ to be the total number of votes, where $|f|=f(\alpha)+f(-\alpha)$ for all $\alpha\in\Phi$. The \emph{preference} of $f$ is $I(f):=\{\alpha\in \Phi \mid f(\alpha)\ge |f|/2\}$. The \emph{result} of such $f$ under majority rule is \[R(f):=\{w\in W\mid I(w) \subseteq I(f)\}.\] We say that a voting profile $f$ has \emph{no tie} if for each $\alpha\in\Phi$, $f(\alpha)\neq|f|/2$. Our first theorem below justifies our generalizations and serves as the bases for further discussions.
\begin{theorem}\label{thm: fundamental}
Let $f:D\rightarrow\Z_{\geq0}$ be a voting profile on a Condorcet domain $D$. Then the result $R(f)$ under majority rule is nonempty. Moreover, if $f$ has no tie, then there exists $w\in W$ such that $I(w)=I(f)$, and $R(f)=\{w\}$ is a singleton. 
\end{theorem}

Our main tool to analyze the Coxeter Condorcet domains is the \emph{Condorcet root (pre)poset}, which is a family of transitive orders on $\Phi$ satisfying certain conditions on rank $2$ root subsystems (see Definition~\ref{def: condorcet-poset} in \Cref{subsec: basic setting}). These posets have natural relations with Condorcet domains. Let $\mathcal{D}$ be the set of all Condorcet domains and let $\mathcal{P}$ be the set of all Condorcet root (pre)posets.
\begin{theorem}\label{thm: bijection}
There are explicit maps $\varphi:\mathcal{D}\rightarrow\mathcal{P}$ and $\psi:\mathcal{P}\rightarrow\mathcal{D}$ providing a pair of mutually inverse bijections between $\psi(\mathcal{P})\subseteq \mathcal{D}$ and $\varphi(\mathcal{D})=\mathcal{P}$. 
\end{theorem}
The explicit maps $\varphi$ and $\psi$ will be provided in \Cref{subsec: Relations with Condorcet domains}. See Figure~\ref{fig:maps-cartoon} for a visualization. This language allows us to efficiently characterize many desirable properties on Condorcet domains. One example is that of a closed Condorcet domain. Define 
$$\nu(D):=\bigcup_{\substack
{f:D\rightarrow \Z_{\ge0}\\ f \text{ has no tie}}}R(f)$$
 to be all possible results of no tie voting profile on a Condorcet domain $D$. Note that $D\subseteq \nu(D)$. A Condorcet domain is \emph{closed} if $\nu(D)=D$. Let $\overline{D}$ denote the \emph{closure} of $D$, i.e., the minimal closed Condorcet domain containing $D$. Our characterization of closed Condorcet domains is novel even in type $A$.
\begin{cor}\label{cor:bijection}
The closed Condorcet domains are precisely the image $\psi(\mathcal{P})$, and they are in bijection with $\mathcal{P}$ under $\varphi$ with inverse $\psi$. Moreover, for $D_1, D_2\in \mathcal{D}$, $\varphi(D_1)=\varphi(D_2)$ if and only if $\overline{D}_1=\overline{D}_2$.
\end{cor}
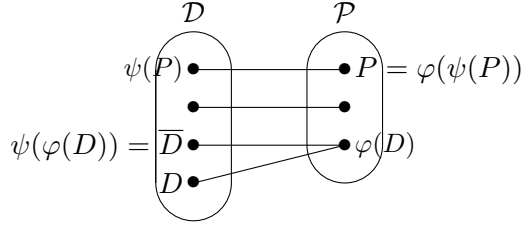
\begin{figure}
\centering
\begin{tikzpicture}[scale=0.5]
\node at (0,0) {$\bullet$};
\node at (0,-1) {$\bullet$};
\node at (0,-2) {$\bullet$};
\node at (0,-3) {$\bullet$};
\draw(-1,0)--(-1,-3);
\draw(1,0)--(1,-3);
\node[left] at (0,0) {\small{$\psi(P)$}};
\node[left] at (0,-2) {$\psi(\varphi(D))=\overline D$};
\node[left] at (0,-3) {$D$};
\draw (1,0) arc (0:180:1);
\draw (-1,-3) arc (180:360:1);
\node[above] at (0,1) {$\mathcal{D}$};

\node at (4,0) {$\bullet$};
\node at (4,-1) {$\bullet$};
\node at (4,-2) {$\bullet$};
\draw(3,0)--(3,-2);
\draw(5,0)--(5,-2);
\draw(-1,0)--(-1,-3);
\draw (5,0) arc (0:180:1);
\draw (3,-2) arc (180:360:1);
\node[right] at (4,-2) {\small{$\varphi(D)$}};
\node[right] at (4,0) {$P=\varphi(\psi(P))$};
\draw(0,0)--(4,0);
\draw(0,-1)--(4,-1);
\draw(0,-2)--(4,-2);
\draw(0,-3)--(4,-2);
\node[above] at (4,1) {$\mathcal{P}$};
\end{tikzpicture}
\caption{Visualization of the maps $\varphi:\mathcal{D}\rightarrow\mathcal{P}$ and $\psi:\mathcal{P}\rightarrow\mathcal{D}$}
\label{fig:maps-cartoon}
\end{figure}

To further understand the structure of generalized Condorcet domains, we explore their connection to median graphs. Originating from the study of distributive lattices and characterized by the property that any three vertices have a unique median on their shortest paths, median graphs play a central role in metric graph theory (such as in \cite{mulder1980interval, klavzar1999median, bandelt2008metric}). Previous studies in \cite{Puppe-Slinko-median-graph} have established a profound link in type $A$, showing that closed Condorcet domains in type $A$ can induce median graphs. We extend this elegant correspondence to arbitrary finite Coxeter groups. We define the associated graph $G_D$ in Section~\ref{sec: median graph} for a Condorcet domain $D$, which generalizes the definition in \cite[Section 3.2]{Puppe-Slinko-median-graph}.
\begin{theorem}
    \label{thm: Condorcet domain to median graph}
    If $D$ is a closed Condorcet domain, then $G_D$ is a median graph. 
\end{theorem}

Conversely, it was proved in \cite[Theorem 5]{Puppe-Slinko-median-graph} that every median graph is associated with a closed Condorcet domain in type $A$. But in general, such Condorcet domain is not unique. So it's natural to consider the conditions under which two distinct Condorcet domains induce the same median graph. We characterize such a condition using skeleton isomorphism of Condorcet root posets.
\begin{theorem}
    \label{thm: eq median graph iff eq skeleton}
    Let $D$ and $D'$ be two closed Condorcet domains. Then $G_D$ and $G_{D'}$ are isomorphic if and only if $\varphi(D)$ and $\varphi(D')$ are skeleton isomorphic.
\end{theorem}

A central problem in the theory of Condorcet domains is determining their maximum possible sizes. To this end, various structural properties, such as being ample, connected, peak-pit, or of tiling type, have been used to study their relationships with maximal and maximum-size Condorcet domains. In Section~\ref{sec: families}, we generalize these definitions to arbitrary finite Coxeter groups and establish equivalent characterizations in the language of Condorcet root posets. To connect these properties with maximality, we provide a combinatorial criterion for maximal Condorcet domains.

\begin{theorem}
    \label{thm: minimal cover is root-triple}
    Let $D$ be a connected, ample and closed Condorcet domain and $P=\varphi(D)$. Then $D$ is a maximal Condorcet domain if and only if for any covering relation $\alpha\lessdot \beta$ in $P$, there exists a root-triple containing $\alpha$ and $-\beta$.
\end{theorem}

Using those equivalent characterizations and the above criterion, we generalize several classical theorems concerning these properties to the broader Coxeter group setting. For instance, the following theorem generalizes a recent result for type $A$ from \cite{liEquivalenceConnectedPeakPit2026} and our proof can be viewed as a simplification of their approach.

\begin{theorem}
    \label{thm: connected eq peak-pit}
    Let $D$ be a maximal Condorcet domain. Then $D$ is connected if and only if $D$ is peak-pit.
\end{theorem}

And the next theorem provides a Coxeter-theoretic generalization of \cite[Theorem 4]{danilov2012condorcet}, completely characterizing Condorcet domains of tiling type.

\begin{theorem}
    \label{thm: tiling type}
    Let $D$ be a closed Condorcet domain. Then $D$ is of tiling type if and only if $D$ is connected, of maximal width, and maximal, with $e, w_0\in D$.
\end{theorem}

When defining the majority outcome $\nu(D)$, we typically require the voting profile $f$ to have no tie in order to guarantee a unique result. However, there are instances where $f$ contains ties, yet the set of all possible outcomes is still contained in $\nu(D)$, meaning $R(f) \subseteq \nu(D)$. In such cases, we say that the voting profile $f$ has only \emph{simple ties}. This relaxation allows Condorcet domains to accommodate voting scenarios where ties are unavoidable.

The following theorem provides a broad sufficient condition for a voting profile to exhibit only simple ties, and generalizes a recent result from \cite[Corollary 1.4]{reiner2025majority}, which originally established this property strictly for domains of tiling type in type $A$.

\begin{theorem}
    \label{thm: simple ties}
    Let $f$ be a voting profile on a connected, closed Condorcet domain $D$. If $\operatorname{supp}(f)=D$, then $f$ has only simple ties.
\end{theorem}

This article is organized as follows. In \Cref{sec: poset}, we introduce the Condorcet root poset and prove Theorem~\ref{thm: fundamental} in \Cref{subsec: basic setting}, and then define the maps $\varphi$ and $\psi$ in \Cref{subsec: Relations with Condorcet domains}. In \Cref{subsec: Antipodal root posets}, we prove some properties of antipodal root posets, which is a generalization of Condorcet root poset, and with these basis, we prove Theorem~\ref{thm: bijection} and Corollary~\ref{cor:bijection} in \Cref{subsec: Closed Condorcet domains}. In \Cref{sec: median graph}, we introduced the concepts of median graphs, associated graphs and skeleton isomorphisms, and prove Theorems~\ref{thm: Condorcet domain to median graph} and ~\ref{thm: eq median graph iff eq skeleton}. In \Cref{sec: families}, we introduce the concept of ample and connected and prove Theorem~\ref{thm: minimal cover is root-triple} in \Cref{subsec: Maximal ample and connected}. In \Cref{subsec: Peak-pit and saturated single-crossing}, we introduce peak-pit and saturated single-crossing domains and use it to prove Theorem~\ref{thm: connected eq peak-pit}. In \Cref{subsec: Symmetry maximal width and tiling type}, we introduce concepts of symmetry, maximal width and tiling type, and prove Theorem~\ref{thm: tiling type}. In \Cref{subsec: Restriction to root subsystems}, we simply discuss the restriction of a Condorcet domain into a root subsystems. In \Cref{sec: voting}, we prove Theorem~\ref{thm: simple ties}.

\section{Condorcet root posets}\label{sec: poset}
\subsection{Basic settings}\label{subsec: basic setting}

The following is our key definition, which we use extensively.
\begin{defin}\label{def: condorcet-poset}
A preposet $P$ on $\Phi$ is an \emph{antipodal root poset} if:
\begin{itemize}
\item[(1)] For all $\alpha,\beta\in\Phi,$ $\alpha\leq\beta$ implies $-\beta\leq-\alpha$.
\item[(2)] For all $\gamma\in\Phi,$ $\gamma\neq-\gamma$. If $\gamma< -\gamma$, then $\gamma$ is a minimum element in $P$.
\end{itemize}
An antipodal root poset is a \emph{Condorcet root poset} if it satisfies the additional condition
\begin{itemize}
\item[(3)] For every root-triple, there exists a label $\{\alpha,\beta,\gamma\}$ of its elements such that \[\alpha\ge-\gamma\ge\beta\quad\text{or}\quad\alpha=-\gamma.\]
\end{itemize}
We use $\mathcal{P}$ to denote the set of Condorcet root posets.
\end{defin}
Condition (3) of Definition~\ref{def: condorcet-poset} is the main technical requirement. We call these two cases the \emph{first} and \emph{second type of Condorcet conditions} on the root-triple $\{\alpha, \beta, \gamma\}$, respectively.

Note that in a preposet, there may be multiple minimum elements, which equal to each other. Hence, we use the terminology ``a'' minimum element in Condition (2) to mean that $\gamma$ is one of such minimum element. Write $P_m$ and $P_M$ for the set of minimum and maximum elements in $P$, respectively. Define the \emph{global elements} to be $P_G:=P_m\cup P_M$. Its complement $\tilde{P}:=P\setminus P_G$ is called the \emph{skeleton} of $P$. To avoid ambiguity, we use the notation $\alpha =_\Phi \beta$ to mean that  $\alpha$ and $\beta$ are the same roots, and $\alpha =\beta $ to mean that they are equal in $P$, when $P$ is naturally understood.

Throughout this work, unless specified otherwise, we drop the antisymmetry axiom for \emph{posets}, allowing different elements to be the same in a binary relation. In other words, we will be working with \emph{preposets} most of the times. We will stick with the terminology ``posets" in order to appeal to the broader audience.

There is an equivalent formulation of the Condorcet condition for Condorcet root posets, which proves to be particularly useful in non-simply-laced cases. We say that a sequence $(\gamma_1,\ldots,\gamma_{2m})$ is a \emph{dihedral root cycle} if $\gamma_{i+2}=-s_{\gamma_{i+1}}(\gamma_i)$ for all $i$ (where indices are taken modulo $2m$), and the subspace $U=\mathrm{span}\{\gamma_1,\ldots,\gamma_{2m}\}$ is two-dimensional with $\Phi\cap U=\{\gamma_1,\ldots,\gamma_{2m}\}$. The set of all dihedral root cycles for a fixed rank-$2$ root subsystem exhibits dihedral symmetry. As such, the indices of a dihedral root cycle of size $2m$ are always considered modulo $2m$. Note that from definition of dihedral root cycle, we have $\gamma_{i+m}=-\gamma_i$ for all $i$. See \Cref{fig:condorcet-root-cycle} for an example. 

\begin{figure}[h!]
\centering
\begin{tikzpicture}[scale=1.0]
\def\r{1};
\draw[->](0,0)--(36:\r);
\draw[->](0,0)--(72:\r);
\draw[->](0,0)--(108:\r);
\draw[->](0,0)--(144:\r);
\draw[->](0,0)--(180:\r);
\draw[->](0,0)--(216:\r);
\draw[->](0,0)--(252:\r);
\draw[->](0,0)--(288:\r);
\draw[->](0,0)--(324:\r);
\draw[->](0,0)--(360:\r);
\node at (36:\r+0.3) {$\gamma_1$};
\node at (72:\r+0.3) {$\gamma_2$};
\node at (108:\r+0.3) {$\gamma_3$};
\node at (144:\r+0.3) {$\gamma_4$};
\node at (180:\r+0.3) {$\gamma_5$};
\node at (216:\r+0.3) {$\gamma_6$};
\node at (252:\r+0.3) {$\gamma_7$};
\node at (288:\r+0.3) {$\gamma_8$};
\node at (324:\r+0.3) {$\gamma_9$};
\node at (0:\r+0.3) {$\gamma_{10}$};
\node[rotate=54-90] at (54:\r) {$\geq$};
\node[rotate=18-90] at (18:\r) {$\geq$};
\node[rotate=-18-90] at (-18:\r) {$\geq$};
\node[rotate=-54-90] at (-54:\r) {$\geq$};
\node[rotate=126+90] at (126:\r) {$\geq$};
\node[rotate=162+90] at (162:\r) {$\geq$};
\node[rotate=198+90] at (198:\r) {$\geq$};
\node[rotate=234+90] at (234:\r) {$\geq$};
\end{tikzpicture}
\quad
\begin{tikzpicture}[scale=1.0]
\def\r{1};
\draw[->](0,0)--(36:\r);
\draw[->](0,0)--(72:\r);
\draw[->](0,0)--(108:\r);
\draw[->](0,0)--(144:\r);
\draw[->](0,0)--(180:\r);
\draw[->](0,0)--(216:\r);
\draw[->](0,0)--(252:\r);
\draw[->](0,0)--(288:\r);
\draw[->](0,0)--(324:\r);
\draw[->](0,0)--(360:\r);
\node at (36:\r+0.3) {$\gamma_1$};
\node at (72:\r+0.3) {$\gamma_2$};
\node at (108:\r+0.3) {$\gamma_3$};
\node at (144:\r+0.3) {$\gamma_4$};
\node at (180:\r+0.3) {$\gamma_5$};
\node at (216:\r+0.3) {$\gamma_6$};
\node at (252:\r+0.3) {$\gamma_7$};
\node at (288:\r+0.3) {$\gamma_8$};
\node at (324:\r+0.3) {$\gamma_9$};
\node at (0:\r+0.3) {$\gamma_{10}$};
\node[rotate=54-90] at (54:\r) {$=$};
\node[rotate=90-90] at (90:\r) {$=$};
\node[rotate=162-90] at (162:\r) {$=$};
\node[rotate=234-90] at (234:\r) {$=$};
\node[rotate=270-90] at (270:\r) {$=$};
\node[rotate=342-90] at (342:\r) {$=$};
\end{tikzpicture}
\caption{Examples of the first and second type of Condorcet conditions for $m=5$ (c.f. Theorem~\ref{thm: cycle condition})}
\label{fig:condorcet-root-cycle}
\end{figure}
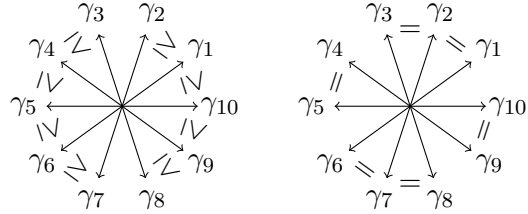

\begin{theorem}\label{thm: cycle condition}
    Let $P$ be an antipodal root poset on $\Phi$. The following conditions are equivalent:
    \begin{itemize}
        \item[(1)] The Condorcet condition holds for $P$, that is, $P$ is a Condorcet root poset.
        \item[(2)] For every dihedral root cycle $(\gamma_1,\ldots,\gamma_{2m})$, either there exists an index $i$ such that $\gamma_i\geq \gamma_{i+1}\geq\cdots\geq\gamma_{i+m-1}$ in $P$, or there exist indices $i$ and $j$ with $i<j<i+m$ such that $\gamma_i =\gamma_{i+1}=\cdots=\gamma_{j-1}$ and $\gamma_j=\gamma_{j+1}=\cdots=\gamma_{i+m-1}$ in $P$. 
    \end{itemize}
\end{theorem}
A visualization is provided in \Cref{fig:condorcet-root-cycle}. We refer to the two cases in condition (2) as the \emph{first} and \emph{second type of cycle conditions} on $(\gamma_1,\ldots,\gamma_{2m})$, respectively.

Before proceeding to the proof of Theorem~\ref{thm: cycle condition}, we require some additional definitions. Let $(\gamma_1, \ldots, \gamma_{2m})$ be a dihedral root cycle. Choose an integer $1\le k\le m$ and a sequence of $k$ indices $i_1< i_2< \cdots < i_k< i_1+m$. We define the sequence 
\[
(\gamma'_1, \ldots, \gamma'_{2k}) := (\gamma_{i_1}, \ldots , \gamma_{i_k}, \gamma_{i_1+m}, \ldots, \gamma_{i_k+m})
\]
to be a \emph{partial root cycle} of size $2k$. As with full root cycles, the set of all partial root cycles of size $2k$ exhibits dihedral symmetry, so their indices are considered modulo $2k$. Note that $\gamma'_{i+k}=-\gamma'_i$ for all $i$. 

We define analogous conditions for any partial root cycle $(\gamma'_1, \ldots, \gamma'_{2k})$: either there exists an index $i$ such that $\gamma'_i\geq \gamma'_{i+1}\geq\cdots\geq\gamma'_{i+k-1}$, or there exist indices $i$ and $j$ with $i<j<i+k$ such that $\gamma'_i =\gamma'_{i+1}=\cdots=\gamma'_{j-1}$ and $\gamma'_j=\gamma'_{j+1}=\cdots=\gamma'_{i+k-1}$. We similarly refer to these as the \emph{first} and \emph{second type of cycle conditions on $(\gamma'_1, \ldots, \gamma'_{2k})$}. It is straightforward to verify that if a dihedral root cycle satisfies the cycle conditions, then every partial root cycle derived from it satisfies them as well.
\begin{proof}[Proof of Theorem~\ref{thm: cycle condition}]
    We first prove the implication $(1)\Rightarrow (2)$. Assuming $P$ is a Condorcet root poset, we prove by induction on $k$ that the cycle conditions hold on every partial root cycle $(\gamma'_1, \ldots, \gamma'_{2k})$ for all $1\le k\le m$.
        
    For $k=1$ and $k=2$, the conditions are trivial and there is nothing to prove. For $k=3$, note that $\{\gamma'_1, \gamma'_3, \gamma'_5\}$ forms a root-triple. The two types of Condorcet conditions on the root-triple $\{\gamma'_1, \gamma'_3, \gamma'_5\}$ precisely correspond to the two types of cycle conditions on the partial root cycle $(\gamma'_1, \ldots, \gamma'_{6})$. 
    
    For $k\ge 4$, suppose there exists an index $i$ such that $\gamma'_i=\gamma'_{i+1}$. Without loss of generality, we may assume $i=k-1$. Then the cycle condition on the smaller partial root cycle 
    $(\gamma'_1,\ldots,\gamma'_{k-1}, \gamma'_{k+1},\ldots,\gamma'_{2k-1})$
    immediately implies the corresponding cycle condition on $(\gamma'_1,\gamma'_2,\ldots,\gamma'_{2k})$ in both types. Therefore, by induction, we may assume throughout the remainder of the proof that there are no adjacent equal roots.

    Suppose the induction hypothesis holds for $k=l$, where $3\le l<m$. We will prove it holds for $k=l+1$. Let $C=(\gamma'_1, \gamma'_2,\ldots, \gamma'_{2l+2})$ be any partial root cycle of size $2l+2$. We claim that the first type of cycle condition holds for every partial root cycle of size $2l$ contained in $C$. If $l\ge 4$, this is immediately true since there are no adjacent equal roots. If $l=3$, by dihedral symmetry, we may assume without loss of generality that the second type of cycle condition holds on $(\gamma'_1, \gamma'_2, \gamma'_3, \gamma'_5, \gamma'_6,\gamma'_7)$, which implies $\gamma'_3=\gamma'_5$. 
    
    Now consider the partial root cycle $(\gamma'_3, \gamma'_4, \gamma'_5, \gamma'_7, \gamma'_8, \gamma'_1)$. Suppose the first type of cycle condition holds on it. Since $\gamma'_3=\gamma'_5$, the strict chains $\gamma'_3<\gamma'_4<\gamma'_5$ and $\gamma'_3>\gamma'_4>\gamma'_5$ cannot hold. If $\gamma'_1<\gamma'_3<\gamma'_4$ holds, we obtain $\gamma'_1<\gamma'_3=\gamma'_5=-\gamma'_1$, this forces $-\gamma'_1=\gamma'_3\in P_M$, contradicting $\gamma'_3<\gamma'_4$. By applying similar reasoning to the other configurations of the first type of cycle condition, we deduce that none can hold. Thus, only the second type of cycle condition can hold, which yields $\gamma'_1=\gamma'_3$. We then have $\gamma'_1=\gamma'_3=\gamma'_5=-\gamma'_1$, which contradicts the fact that $P$ is an antipodal root poset. Therefore, the claim holds.

    Consider the partial root cycle $(\gamma'_1, \ldots, \gamma'_{l}, \gamma'_{l+2}, \ldots, \gamma'_{2l+1})$ of size $2l$ in $C$. By the first type of cycle condition and dihedral symmetry, we may assume without loss of generality that there are two cases to consider:
    
    \textbf{Case (i) }$\gamma'_1>\gamma'_2>\cdots>\gamma'_l$. Consider the partial root cycle $C'=(\gamma'_{l},\gamma'_{l+1},\gamma'_{l+2}, \gamma'_{2l+1},\gamma'_{2l+2},\gamma'_1)$. The second type of cycle condition cannot hold on $C'$ because there are no adjacent equal roots and $\gamma'_1>\gamma'_{l}$. For the first type of cycle condition to hold, the relation $\gamma'_1>\gamma'_{l}$ implies that either $\gamma'_{l}>\gamma'_{l+1}$ or $\gamma'_{2l+2}>\gamma'_1$ must be true. Either of these options extends the chain, thereby satisfying the first type of cycle condition on the entire partial root cycle of size $2l+2$.
    
    \textbf{Case (ii) }$\gamma'_i>\gamma'_{i+1}>\cdots>\gamma'_l>\gamma'_{l+2}>\gamma'_{l+3}>\cdots>\gamma'_{l+i}$ for some $2 \le i\le l$. We proceed by induction on $i$. Consider the partial root cycle $C'=(\gamma'_{l},\gamma'_{l+1},\gamma'_{l+2}, \gamma'_{2l+1},\gamma'_{2l+2},\gamma'_1)$. Note that if $\gamma'_l>\gamma'_{l+1}$, the chain extends to $\gamma'_i>\gamma'_{i+1}>\cdots>\gamma'_{l+1}>\gamma'_{l+3}>\cdots>\gamma'_{l+i}$. This reduces to Case (i) when $i=2$, or allows us to apply the inductive hypothesis (on $i$) when $i>2$, completing the proof for this scenario. 
    
    Now assume $\gamma'_l\ngtr \gamma'_{l+1}$. If the second type of cycle condition holds on $C'$, we have $\gamma'_1=\gamma'_l$, which leads to $\gamma'_1=\gamma'_l>\gamma'_{l+2}=-\gamma'_1$. Thus $\gamma'_1=\gamma'_l\in P_M$, which forces $\gamma'_l > \gamma'_{l+1}$, contradicting the assumption that $\gamma'_l\ngtr \gamma'_{l+1}$. Therefore, only the first type of cycle condition can hold on $C'$. If we assume $\gamma'_{l+2}>\gamma'_{2l+1}$, then $\gamma'_l>\gamma'_{l+2}>\gamma'_{2l+1}=-\gamma'_l$, meaning $\gamma'_l\in P_M$, which again yields the same contradiction. Finally, assuming both $\gamma'_l\ngtr \gamma'_{l+1}$ and $\gamma'_{l+2}\ngtr \gamma'_{2l+1}$, the relation $\gamma'_l>\gamma'_{l+2}$ forces the first type of cycle condition on $C'$ to be precisely $\gamma'_{l+1}<\gamma'_{l+2}<\gamma'_{2l+1}$. This yields $\gamma'_l>\gamma'_{l+2}>\gamma'_{l+1}$, which contradicts $\gamma'_l\ngtr \gamma'_{l+1}$, completing the proof.
    
    We now prove the implication $(2)\Rightarrow (1)$. Let $\{\alpha, \beta, \gamma\}$ be any root-triple. Since these roots are coplanar, the subspace $U=\mathrm{span}\{\alpha, \beta, \gamma\}$ is a two-dimensional plane. The intersection $\Phi\cap U$ defines a dihedral root cycle $(\gamma_1,\ldots,\gamma_{2m})$, and there exist indices $i$, $j$, and $k$ such that $\alpha=\gamma_i$, $\beta=\gamma_j$, and $\gamma=\gamma_k$, with $j<i+m$, $k<j+m$, and $i<k+m$. 
    
    By the cycle condition on $(\gamma_1,\ldots,\gamma_{2m})$, either there exists an index $i'$ such that $\gamma_{i'}\geq \gamma_{i'+1}\geq\cdots\geq\gamma_{i'+m-1}$ or there exist indices $i'$ and $j'$ with $i'<j'<i'+m$ such that $\gamma_{i'} =\gamma_{i'+1}=\cdots=\gamma_{j'-1}$. After appropriately relabeling the roots of the root-triple, the sequence $(\alpha, -\gamma, \beta)$ must appear in relative order within either $(\gamma_{i'}, \gamma_{i'+1}, \ldots ,\gamma_{i'+m-1})$ or $(\gamma_{i'+m}, \gamma_{i'+m+1}, \ldots ,\gamma_{i'+2m-1})$. The cycle condition restricted to them immediately yields the Condorcet condition on $\{\alpha, \beta, \gamma\}$.

\end{proof}

Before connecting Condorcet root posets with Condorcet domains, we first establish the bases for generalizing the Condorcet domain to Coxeter groups.

\begin{defin} \label{def: biconvex}
A subset $I\subseteq \Phi$ is \emph{biconvex} if $|\{\alpha,-\alpha\}\cap I|=1$ for all $\alpha\in\Phi$, and $|\{\alpha,\beta,\gamma\}\cap I|\in\{1,2\}$ for all root-triples $\{\alpha,\beta,\gamma\}$.
\end{defin}
Note that our definitions here differ slightly from the literature, as we do not require $I\subseteq \Phi^+$. The following classical result is very useful, see for example \cite[Lemma 4.1]{Dyer-weak}.
\begin{prop}\label{prop: biconvex}
A subset $I\subseteq \Phi$ is the inversion set $I(w)$ for some $w\in W$ if and only if $I$ is biconvex.
\end{prop}
Since we are working in the generality of finite Coxeter groups, we do not need to specify that $I\cap\Phi^+$ is finite. We can now prove Theorem~\ref{thm: fundamental}.

\begin{proof}[Proof of Theorem~\ref{thm: fundamental}.] 
We first consider the case when $f$ has no tie. In other words, for every $\alpha \in \Phi,\ |\{\alpha, -\alpha\}\cap I(f)|=1$. We only need to prove that $I(f)$ is biconvex by Proposition~\ref{prop: biconvex}. For any root-triple $\{\alpha,\beta,\gamma\}$, if $\alpha, \beta, \gamma \in I(f)$, then $f(\alpha), f(\beta), f(\gamma) > |f|/2.$ As $f(\alpha)+f(\beta)>|f|, $ there exist $u\in D$ with $\alpha, \beta \in I(u)$. Similarly, there are also $v, w\in D$ with $\beta, \gamma \in I(v) $ and $\alpha, \gamma \in I(w)$, contradicting the Condorcet condition on $D$. If $\alpha, \beta, \gamma \notin I(f)$, then we have $-\alpha, -\beta, -\gamma\in I(f)$ for root-triple $\{-\alpha, -\beta, -\gamma\}$. Similarly, there is a contradiction, too. So $|\{\alpha,\beta,\gamma\}\cap I(f)|\in\{1,2\}$, hence $I(f)$ is biconvex.

When there is a tie in $f$, $|f|$ must be even. Let $|f|=2m$. Fix any $v\in D$. We take a voting profile $f'$ with $f'(u)=f(u)$ for $u\neq v$ and $f'(v)= f(v)+1$. As $|f'|=2m+1$ is odd, $f'$ has no tie. As $I(f')=\{\alpha\in \Phi\mid f'(\alpha)\ge m+1\} \subseteq \{\alpha\in \Phi \mid f(\alpha)\ge m\}=I(f)$, $R(f')\subseteq R(f)$. By the previous paragraph, $R(f')\neq\emptyset$ so $R(f)\neq\emptyset$.
\end{proof}

\begin{remark} \label{rmk: Condorcet domain}
The converse of Theorem~\ref{thm: fundamental} is true. If $D$ is not a Condorcet domain, then there exists $u,v,w\in D$ and a root-triple $\{\alpha,\beta,\gamma\}$ such that $\alpha,\beta\in I(u)$, $\alpha,\gamma\in I(v)$, $\beta,\gamma\in I(w)$. Take a voting profile $f$ on $D$ with value $1$ on $u, v, w$ and $0$ elsewhere. Then $\{\alpha, \beta, \gamma\}\subseteq I(f)$ and $f$ has no tie. By Proposition~\ref{prop: biconvex}, $R(f)$ must be empty. This is just the generalization of Condorcet paradox in Coxeter cases.
\end{remark}

In this article, when consider examples in type $A_{n-1}$, we take the convention $\Phi^+=\{e_i-e_j\mid1\le i< j\le n\}$, and use the shorthand $ij$ for the root $e_i-e_j$ and $-ij$ for the root $e_j-e_i$ for $1\le i<j\le n$. Now we have some basic examples of Condorcet domains and voting profiles on it.

\begin{ex}\label{ex: Condorcet and voting}
\begin{enumerate}
\item In type $A_2$ with Weyl group $D=S_3$. Let $D=\{123, 231, 312\}$. We have $I(123)=\{-12, -23, -13\}, I(231)=\{12,-23, 13\}, I(312)=\{-12, 13, 23\}$. Take the voting profile $f$ on $D$ with constant value $1$. Then $I(f)=\{-12, 13,-23\}$ is a root-triple, hence not a biconvex set. So $R(f)=\emptyset$. This is the classic Condorcet paradox, which is the origin of all the Condorcet theory.
\item In any finite Coxeter group $W$, $D=\{e, w_0\} $ is always a Condorcet domain as $D$ has only $2$ elements. Take a voting profile $f$ on $D$ with constant value $1$. As $I(e)=\Phi^-, I(w_0)=\Phi^+$, we have $I(f)=\Phi$ and $R(f)=W$. This example illustrates the no tie condition is indispensable.
\item In type $A_2$ with Weyl group $W=S_3$. Let $D=\{213, 231,321\}$, we have $I(213)=\{12, -23, -13\}, I(231)=\{12,-23, 13\}, I(321)=\{12, 13, 23\}$. The only two root-triples in $A_2$ are $\{12, 23, -13\}$ and $\{-12, -23, 13\}$. We can check $D$ is a Condorcet domain by checking Definition~\ref{def: condorcet} on these root-triples. Take a no tie voting profile $f$ on $D$ with $f(321)=a, f(231)=b, f(213)=c, abc\neq 0$. Then $f(23)=a, f(-23)=b+c, f(13)=a+b, f(-13)=c, f(12)=a+b+c=|f|, f(-12)=0$. If $a>b+c$, then $I(f)=\{23, 13, 12\}=I(321), R(f)=\{321\}$. Similarly, if $c>a+b$, then $R(f)=\{213\}$. If $a<b+c$ and $c<a+b$, then $R(f)=\{231\}$. We can see in each case, $R(f)$ is a singleton, as Theorem~\ref{thm: fundamental} says. Moreover, we have $R(f)\subseteq D$ hence $\nu(D)\subseteq D$.
\end{enumerate}
\end{ex}
The notion of closed Condorcet domain is another essential concept.

\begin{defin}
    \label{def: closed}
    A Condorcet domain $D$ is \emph{maximal} if for any Condorcet domain $D'$, the inclusion $D'\supseteq D$ implies $D'=D$. It is \emph{closed} if $\nu(D)\subseteq D$. For a Condorcet domain $D$, the \emph{closure} of $D$ is defined as the minimal closed Condorcet domain containing $D$. 
\end{defin}

\begin{remark}\label{rmk: closed}
    For any $v\in D$, consider a voting profile $f$ on $D$ where $f(v)=1$ and $f(w)=0$ for all $w\neq v$. Then $f$ has no tie and $R(f)=\{v\}\subseteq \nu(D)$. Therefore, $D\subseteq \nu(D)$, which implies that $D$ is closed if and only if $\nu(D)=D$.
\end{remark}

For instance, in Example~\ref{ex: Condorcet and voting} (3), we demonstrated that the domain $D=\{321, 231, 213\}$ in type $A_2$ is a closed Condorcet domain.

\begin{lemma}
    \label{lem: closure}
    If $D$ and $D'$ are closed Condorcet domains, then $D\cap D'$ is also a closed Condorcet domain. Consequently, $\overline{D}$ is well-defined and $\overline{D}$ is the intersection of all closed Condorcet domain containing $D$.
\end{lemma}
\begin{proof}
    By Definition~\ref{def: condorcet}, any subset of a Condorcet domain is a Condorcet domain, so $D\cap D'$ is a Condorcet domain. Let $v\in \nu(D\cap D')$. Then there exists a voting profile $f$ with no tie on $D\cap D'$ such that $R(f)=\{v\}$. $f$ can also be viewed as a valid voting profile with no tie on $D$ and on $D'$. Therefore, $v\in \nu(D)$ and $v\in \nu(D')$. Because $D$ and $D'$ are closed, $\nu(D)\subseteq D$ and $\nu(D')\subseteq D'$, yielding $v\in D\cap D'$. This shows that $\nu(D\cap D')\subseteq D\cap D'$, so $D\cap D'$ is closed. 
    
    Because the intersection of any finite family of closed Condorcet domains is closed and $W$ is finite, the intersection of all closed Condorcet domain containing $D$ is still a closed Condorcet domain. It follows immediately from Definition~\ref{def: closed} that $\overline{D}$ is just the intersection of all closed Condorcet domain containing $D$.
\end{proof}

\begin{prop}
    \label{prop: closed}
    If $D$ is a Condorcet domain and $v\in \nu(D)$, then $D\cup \{v\}$ is also a Condorcet domain. In particular, every maximal Condorcet domain is closed, and every Condorcet domain is contained in a closed Condorcet domain.
\end{prop}
\begin{proof}
    Suppose, for the sake of contradiction, that $D\cup \{v\}$ is not a Condorcet domain. Then there exist a root-triple $\{\alpha, \beta, \gamma\}$ and elements $u, w\in D$ such that $\alpha, \beta\in I(v)$, $\beta,\gamma\in I(u)$, and $\alpha, \gamma\in I(w)$. Since $v\in \nu(D)$, by Theorem~\ref{thm: fundamental}, there exists a voting profile $f$ with no tie on $D$ such that $I(f)=I(v), R(f)=\{v\}$. As $f$ has no tie, $f(\alpha)> |f|/2$ and $f(\beta)> |f|/2$. Consequently, we have $f(\alpha)+f(\beta)>|f|$, which implies that there exist $v'\in D$ with $\alpha, \beta\in I(v')$. However, the presence of $v', u, w\in D$ directly violates the Condorcet condition for $D$, a contradiction. Thus, $D\cup \{v\}$ must be a Condorcet domain.

    If $D$ is maximal, because $D \cup \{v\}$ is a Condorcet domain for any $v \in \nu(D)$, it follows that $D \cup \{v\} = D$. This forces $\nu(D) \subseteq D$, meaning $D$ is closed. Finally, since $W$ is finite, every Condorcet domain is contained in a maximal Condorcet domain, and therefore contained in a closed Condorcet domain.
\end{proof}
\subsection{Relations with Condorcet domains} \label{subsec: Relations with Condorcet domains}
We define maps between the set of nonempty subsets of $W$ and antipodal root posets, with particularly good behavior on Condorcet domains $\mathcal{D}$ and Condorcet root posets $\mathcal{P}$.
\begin{defin}\label{def: map phi}
For a nonempty subset $D\subseteq W$ and any root $\gamma\in \Phi$, let $D(\gamma):= \{v\in D\mid \gamma\in I(v)\}$. Define a (pre)poset $\varphi(D)$ on $\Phi$ with the relation $\alpha\leq \beta$ if $D(\alpha)\supseteq D(\beta).$  
\end{defin}
\begin{lemma}\label{lem:phi well-defined}
The map $\varphi$ sends a nonempty subset $D\subseteq W$ to an antipodal root poset, and especially sends a Condorcet domain to a Condorcet root poset. In other words, $\varphi:\mathcal{D}\rightarrow\mathcal{P}$ is well-defined.
\end{lemma}
\begin{proof}
For a nonempty $D\subseteq W$, the antipodal conditions in Definition~\ref{def: condorcet-poset} are clear as $D(\alpha)\sqcup D(-\alpha)=D$ for every $\alpha\in\Phi$ and $D$ is nonempty. To show that $\varphi(D)$ is a Condorcet root poset for $D\in\mathcal{D}$, we need to prove condition (3).

Fix any root-triple $\{\alpha,\beta,\gamma\}$. By Definition~\ref{def: biconvex} and Proposition~\ref{prop: biconvex}, for any $v\in D,\ |\{\alpha, \beta, \gamma \}\cap I(v)|\in\{1,2\}$. Consider the family of sets $T=\{\{\alpha,\beta,\gamma\}\cap I(v)\mid v\in D\}$, and let $T_i=\{A\in T\mid |A|=i\}$ for $i\in\{1,2\}$. By Definition~\ref{def: condorcet}, $|T_2|\leq 2$, so we assume without loss of generality that $\{\alpha,\gamma\}\notin T_2$. This means that for $v\in D$, if $\alpha\in I(v)$, then $\gamma\notin I(v)$, $-\gamma\in I(v)$, so $\alpha\geq-\gamma$ in $\varphi(D)$. By considering the root-triple $\{-\alpha,-\beta,-\gamma\}$ with Definition~\ref{def: condorcet}, we also know that $|T_1|\leq 2$. If $\{\alpha\}\notin T_1$, then for $v\in D$, $-\gamma\in I(v)$ implies that $-\beta\notin I(v)$, $\beta\in I(v)$ so $-\gamma\geq\beta$. Overall, $\alpha\geq-\gamma\geq\beta$ as desired. The situation is similar if $\{\gamma\}\notin T$, where we deduce that $\gamma\geq-\alpha\geq\beta$. And if $\{\beta\}\notin T_1$, we have that $-\gamma\geq\alpha$, which means that $\alpha=-\gamma$ as desired.
\end{proof}

\begin{ex}\label{ex:map-varphi-small-example}
Consider the Condorcet root domain $D=\{213,231,321\}$ in type $A_2$, which also appear in Example~\ref{ex: Condorcet and voting} (3). We then compute
\[\begin{tabular}{lll}
$D(12)=\{213,231,321\}$, & $D(13)=\{231,321\}$, & $D(23)=\{321\}$, \\
$D(-12)=\emptyset$, & $D(-13)=\{213\}$, & $D(-23)=\{213,231\}$.
\end{tabular}\]
Therefore, we obtain the Condorcet root poset $\varphi(D)$ shown in Figure~\ref{fig:map-varphi-small-example}.
\begin{figure}[htbp]
    \centering
    \resizebox{0.3\textwidth}{!}{
        \begin{tikzpicture}[
            x=0.7cm, 
            y=0.7cm, 
            every node/.style={
                rectangle, 
                rounded corners=4pt,
                draw, 
                fill=white, 
                inner sep=2pt, 
                minimum width=1.1cm, 
                minimum height=0.6cm
            }
        ]
\def\a{0.7};
\def\b{0.4};
\def\h{2.0};
\def\r{0.2};
\newcommand\Rec[3]{
\node (#3) at (#1,#2) {$#3$};
}
\Rec{0}{0}{12};
\Rec{-4*\a}{\h}{-23};
\Rec{4*\a}{\h}{13};
\Rec{-4*\a}{2*\h}{-13};
\Rec{4*\a}{2*\h}{23};
\Rec{0}{3*\h}{-12};
\draw(12)--(13);
\draw(12)--(-23);
\draw(13)--(23);
\draw(-23)--(-13);
\draw(-13)--(-12);
\draw(23)--(-12);
\end{tikzpicture}
}
\vspace{0.5em}
\caption{The small Condorcet root poset $\varphi(D)$ in Example~\ref{ex:map-varphi-small-example}}
\label{fig:map-varphi-small-example}
\end{figure}
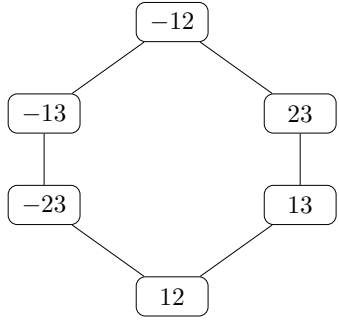
\end{ex}

\begin{remark}
From the proof of Lemma~\ref{lem:phi well-defined}, we can see the set $T$ (hence the Condorcet condition) can be used to characterize a Condorcet domain. In type $A$, the analog is the \emph{never condition} (\cite{monjardet2009acyclic,puppe2024maximal}). We view a Condorcet domain $D\subseteq A_{n-1}$ as a set of permutations on $[n]$, and call the elements in $[n]$ candidates. For any $3$ candidates $i,j,k$, $x\in \{i, j, k\},\text{ and } y\in [3]$, the never condition $xN_{\{i,j,k\}}y$ denotes that in any elements of $D$, $x$ never take the $y^{\text{th}}$ position in $\{i, j, k\}$. Any Condorcet domain must satisfy one never condition on every triple of candidates $\{i, j, k\}$ \cite[Proposition 2.2]{puppe2024maximal}. By checking case by case, the never condition on $\{i, j, k\}$ turns out to be equivalent to the Condorcet condition on root-triple $\{e_i-e_j, e_j-e_k, e_k-e_i\}$, in which the never conditions of from $xN2$ correspond to the second type of Condorcet conditions. For example, in type $A_2$ with Weyl group $S_3$, we take the root-triple $\{12, 23, -13\}$. The Condorcet domain $D=\{123, 132, 231, 321\}$ satisfies the never condition $1N_{\{1,2,3\}}2$, while $\varphi(D)$ satisfies the Condorcet condition $12=-(-13)$.
\end{remark}

\begin{lemma}\label{lem: inversions are symmetric order ideals}
For a nonempty $D\subseteq W$ and $v\in D$, $I(v)$ is an order ideal in $\varphi(D)$ with $I(v)\sqcup -I(v)=\Phi$.
\end{lemma}
\begin{proof}
Take $\alpha\in I(v) \text{ and }  \beta\le \alpha$ in $\varphi(D)$. Then $D(\alpha)\subseteq D(\beta)$, which implies $\beta \in I(v)$. So $I(v)$ is an order ideal in $\varphi(D)$ and $I(v)\sqcup -I(v)=\Phi$ by the property of inversion sets.
\end{proof}
We study Condorcet root posets via their order ideals.
\begin{defin}\label{def: symmetric order ideal}
For an antipodal root poset $P$, an order ideal $I\subseteq P$ is called a \emph{symmetric order ideal} if $I\sqcup-I=P$, and is called a \emph{partial symmetric order ideal} if $I\cap -I=\emptyset$. Denote by $J(P)$, $J_s(P)$ and $J_p(P)$ the set of order ideals, symmetric order ideals, and partial symmetric order ideals of $P$. 
\end{defin}
We have $J_s(P)\subseteq J_p(P)\subseteq J(P)$. Now we define the other important map.

\begin{defin}\label{def: map psi}
For an antipodal root poset $P$, define $\psi(P)=\{w\in W\mid I(w)\in J_s(P)\}$.
\end{defin}
\begin{lemma}\label{lem:psi well-defined}
The map $\psi$ sends a Condorcet root poset $P$ to a Condorcet domain. In other words, $\psi: \mathcal{P} \rightarrow \mathcal{D}$ is well-defined. Moreover, $|J_s(P)|=|\psi(P)|$ for $P\in\mathcal{P}$.
\end{lemma}
\begin{proof}
We first prove that for every $I\in J_s(P)$, $I$ is biconvex. For any root-triple $\{\alpha, \beta, \gamma\}$, if $\alpha, \beta, \gamma \in I$, then $\alpha\ngeq-\beta,\ \beta\ngeq-\gamma, \gamma\ngeq-\alpha $ in $P$. Contradict with conditions (1) and (3) in Definition~\ref{def: condorcet-poset}. Similarly, $\alpha, \beta, \gamma \notin I$ yields the same contradiction. By Proposition~\ref{prop: biconvex} and the fact that different elements in $W$ have different inversion sets, we have $|J_s(P)|=|\psi(P)|.$

To prove $\psi(P)$ is a Condorcet domain, take any root-triple $\{\alpha, \beta, \gamma\}$. If there are $I_1, I_2, I_3\in J_s(P)$ with $\alpha, \beta\in I_1,\ \beta, \gamma\in I_2,\ \alpha, \gamma\in I_3$, then also $\alpha\ngeq-\beta,\ \beta\ngeq-\gamma, \gamma\ngeq-\alpha $ in $P$. This can't hold via the same reason as above.
\end{proof}
\begin{ex}\label{ex:map-psi-example}
Consider the Condorcet root poset in type $A_3$ shown in Figure~\ref{fig:map-psi-example}. We see that $4132\in \psi(P)$ since $I(w)=\{14,23,24,34,-12,-13\}$ is indeed an symmetric order ideal in $P$. Similar calculation shows $\psi(P)=\{3142,3241,3412,3421,4132,4231,4312,4321\},$ and we can check $\psi(P)$ is indeed a Condorcet domain.
\begin{figure}[htbp]
    \centering
    \resizebox{0.4\textwidth}{!}{
        \begin{tikzpicture}[
            x=0.7cm, 
            y=0.7cm, 
            every node/.style={
                rectangle, 
                rounded corners=4pt, 
                draw, 
                fill=white, 
                inner sep=2pt, 
                minimum width=1.1cm, 
                minimum height=0.6cm
            }
        ]
\def\a{0.7};
\def\b{0.4};
\def\h{2.0};
\def\r{0.2};
\newcommand\Rec[3]{
\node (#3) at (#1,#2) {$#3$};
}
\Rec{0}{0}{13};
\Rec{4*\a}{0}{14};
\Rec{8*\a}{0}{23};
\Rec{12*\a}{0}{24};
\Rec{0}{\h}{12};
\Rec{4*\a}{\h}{-34};
\Rec{8*\a}{\h}{34};
\Rec{12*\a}{\h}{-12};
\Rec{0}{2*\h}{-24};
\Rec{4*\a}{2*\h}{-23};
\Rec{8*\a}{2*\h}{-14};
\Rec{12*\a}{2*\h}{-13};
\draw(13)--(12);
\draw(13)--(-34);
\draw(12)--(-23);
\draw(12)--(-24);
\draw(-34)--(-14);
\draw(-34)--(-24);
\draw(23)--(-12);
\draw(23)--(-34);
\draw(-12)--(-14);
\draw(-12)--(-13);
\draw(24)--(-12);
\draw(24)--(34);
\draw(34)--(-23);
\draw(34)--(-13);
\draw(14)--(34);
\draw(14)--(12);
\end{tikzpicture}
}
\vspace{0.5em}
\caption{The Condorcet root poset in type $A_3$ in Example~\ref{ex:map-psi-example}}
\label{fig:map-psi-example}
\end{figure}
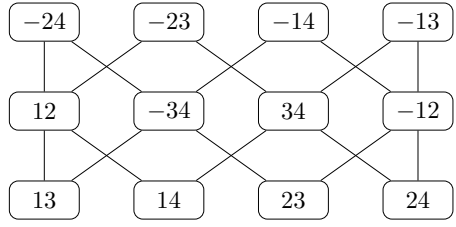
\end{ex}
In fact, the converse of Lemma~\ref{lem:phi well-defined} is also true in the following sense, which provides a fundamental characterization of Condorcet domains by Condorcet root posets.
\begin{prop}\label{prop: Condorcet domain eq map to Condorcet root poset}
Let $D$ be a nonempty subset of $W$. Then $D$ is a Condorcet domain if and only if $\varphi(D)$ is a Condorcet root poset.
\end{prop}
\begin{proof}
If $D$ is a Condorcet domain, then $\varphi(D)$ is a Condorcet root poset by Lemma~\ref{lem:phi well-defined}. Now assume $\varphi(D)$ is a Condorcet root poset. By Lemma~\ref{lem: inversions are symmetric order ideals}, Definition~\ref{def: map psi} and Lemma~\ref{lem:psi well-defined}, $D\subseteq \psi(\varphi(D)) \in \mathcal{D}$. From Definition~\ref{def: condorcet}, a subset of Condorcet domain is still a Condorcet domain, so we are done. 
\end{proof}
\subsection{Antipodal root posets}\label{subsec: Antipodal root posets}
We establish useful properties of antipodal root posets, and eventually prove an analogy of Proposition~\ref{prop: Condorcet domain eq map to Condorcet root poset} on the poset side.

\begin{prop}
\label{prop: J_p(P) in J_s(P)}
Let $P$ be an antipodal root poset. For any order ideal $I$ of $P$, $I\in J_p(P)$ if and only if there exists a $I'\in J_s(P)$ with $I\subseteq I'$. Moreover, $J_s(P)$ is nonempty.
\end{prop}
\begin{proof}
The ``if'' part follows directly from definition. For the ``only if'' part, we take any $I\in J_p(P)$ and assume $I\notin J_s(P)$, which means $P\setminus(I\sqcup-I)\neq \emptyset$. Take a minimal element $\alpha$ in $P\setminus(I\sqcup-I)$, let $A=\{\beta \in P \mid \beta =\alpha\}$, then $I\sqcup A$ is an order ideal with $(I\sqcup A)\cap-(I\sqcup A)=(I\cap -I)\sqcup (A\cap -A)=\emptyset$ as $\alpha\neq -\alpha$ in $P$. So $I\sqcup A\in J_p(P)$. As $P$ is finite, we can repeat this process until $I$ become a symmetric order ideal. As $\emptyset\in J_p(P)$ for any $P$, by the above proof, $J_s(P)$ is nonempty.
\end{proof}
We can recover $P$ from $J(P)$ as $a\le b$ in $P$ if and only if $\{I\in J(P)\mid a\in I\} \supseteq \{I\in J(P)\mid b\in I\}$. Also define $J_\alpha(P):=\{I\in J_s(P)\mid \alpha \in I\}$.
\begin{prop}\label{prop: recover from symmetric order ideal}
For any pair of elements $\alpha$ and $\beta$ in an antipodal root poset $P$, $\alpha\le \beta $ in $P$ if and only if $J_\alpha(P) \supseteq J_\beta(P)$.
\end{prop}
\begin{proof}
By definition of $J_\gamma(P)$, $\alpha \le \beta $ implies $J_\alpha(P)\supseteq J_\beta(P)$. We only need to prove that when $\alpha \nleq \beta$, $J_\alpha(P)\nsupseteq J_\beta(P)$. Let $I$ be the order ideal generated by $-\alpha$ and $\beta$. Then $-I$ is the order filter generated by $\alpha$ and $-\beta$. First assume $I\cap -I\neq \emptyset$. Then one of $-\alpha \ge \alpha$, $-\alpha \ge -\beta$, $\beta \ge \alpha$, and $\beta \ge -\beta$ holds, which contradicts with antipodal conditions in Definition~\ref{def: condorcet-poset} and $\alpha \nleq \beta$. So $I\cap -I= \emptyset$ and $I\in J_p(P)$. By Proposition~\ref{prop: J_p(P) in J_s(P)}, there exists $I' \in J_s(P)$ with $-\alpha, \beta\in I\subseteq I'$. So  $I'\in J_\beta(P)\setminus J_\alpha(P)$ hence $J_\alpha(P)\nsupseteq J_\beta(P)$ and we are done.
\end{proof}
\begin{ex}
    \label{ex: recover from symmetric order ideal}
    Take the Condorcet root poset $P$ in Figure~\ref{fig:map-psi-example} and a symmetric order ideal $I=\{-12,34,-13,24,14,23\}$ of $P$. We can compute
    \[J_{34}(P)=\{ I, \{12,34,13,24,14,23\}, \{-12,34,13,24,14,23\}, \{12,34,13,24,14,-23\} \},\] and $J_{-13}(P)=\{I\}$. So $J_{-13}(P)\subsetneq J_{34}(P)$, corresponding to $-13>34$ in $P$.
\end{ex}
The following theorem characterizes the Condorcet property on posets.
\begin{theorem}
    \label{thm: Condorcet root poset eq condition}
    Let $P$ be an antipodal root poset. Then $P$ is a Condorcet root poset if and only if $\psi(P)$ is a Condorcet domain and $\varphi(\psi(P))=P$. Moreover, $P$ is a Condorcet root poset if and only if $|J_s(P)|=|\psi(P)|$.
\end{theorem}
\begin{proof}
    Suppose that $P$ is a Condorcet root poset. By Lemma~\ref{lem:psi well-defined}, $\psi(P)$ is a Condorcet domain and $|J_s(P)|=|\psi(P)|$. Let $P' = \varphi(\psi(P))$. By Definition~\ref{def: map phi}, for any $\alpha, \beta \in \Phi$, we have $\alpha \le_{P'} \beta$ if and only if $\psi(P)(\alpha) \supseteq \psi(P)(\beta)$. 
    By Definition~\ref{def: map psi}, $\psi(P) = \{w \in W \mid I(w) \in J_s(P)\}$. Consequently, for any $\gamma \in \Phi$, we have 
    \[ \psi(P)(\gamma) = \{v \in W \mid \gamma \in I(v) \text{ and } I(v) \in J_s(P)\} = \{v \in W \mid I(v) \in J_\gamma(P)\}. \]
    Since $|J_s(P)| = |\psi(P)|$, it follows that $|J_\gamma(P)| = |\psi(P)(\gamma)|$. Therefore, $\alpha \le_{P'} \beta$ if and only if $J_\alpha(P) \supseteq J_\beta(P)$. By Proposition~\ref{prop: recover from symmetric order ideal}, the inclusion $J_\alpha(P) \supseteq J_\beta(P)$ is equivalent to $\alpha \le_P \beta$. Thus, $\alpha \le_{P'} \beta$ if and only if $\alpha \le_P \beta$, which implies $P' = P$.
    
    Conversely, if $\psi(P)$ is a Condorcet domain and $\varphi(\psi(P)) = P$, then $P$ is a Condorcet root poset by Lemma~\ref{lem:phi well-defined}.

    Finally, assume that $|J_s(P)| = |\psi(P)|$. To show that $P$ is a Condorcet root poset, we need to prove that $P$ satisfies the Condorcet condition for any root-triple. By Definition~\ref{def: map psi} and Proposition~\ref{prop: biconvex}, any $I \in J_s(P)$ is biconvex. Let $T=\{\alpha, \beta, \gamma\}$ be a root-triple, and suppose for the sake of contradiction that $\zeta_1 \ngeq -\zeta_2$ for all $\zeta_1, \zeta_2 \in T$. Let $I'$ be the order ideal generated by $T$, then $-I' $ is the order filter generated by $-T$. If $I' \cap -I' \neq \emptyset$, there exist $\theta_1, \theta_2 \in T$ such that $\theta_1  \ge -\theta_2$, which is a contradiction. Thus, $I' \in J_p(P)$. By Proposition~\ref{prop: J_p(P) in J_s(P)}, there exists $I \in J_s(P)$ such that $T \subseteq I' \subseteq I$, which contradicts the assumption that $I$ is biconvex. Therefore, there must exist $\zeta_1, \zeta_2 \in T$ such that $\zeta_1 \ge -\zeta_2$.

    If $\zeta_1 \neq \zeta_2$, we may assume without loss of generality that $\alpha \ge -\gamma$. If $\zeta_1 = \zeta_2$, we may assume $\alpha \ge -\alpha$, which means $\alpha \in P_M$ and hence $\alpha \ge -\gamma$. In either case, we may assume $\alpha \ge -\gamma$. Since $-T$ is also a root-triple, there exist $\zeta_1', \zeta_2' \in T$ such that $-\zeta_1' \ge \zeta_2'$. 
    
    If $\zeta_1' = \zeta_2'$, then $\zeta_1' \in P_m$. We consider three subcases:
    \begin{enumerate}
        \item If $\zeta_1' = \alpha$, then $\alpha \in P_m$. The inequality $-\gamma \le \alpha$ implies $-\gamma \in P_m$, so $-\gamma = \alpha$.
        \item If $\zeta'_1 = \beta$, then $\beta \in P_m$. This yields $-\gamma \ge \beta$, and combined with $\alpha \ge -\gamma$, we obtain $\alpha \ge -\gamma \ge \beta$.
        \item If $\zeta_1' = \gamma$, then $-\gamma \in P_M$. The inequality $\alpha \ge -\gamma$ implies $\alpha \in P_M$, so $\alpha = -\gamma$.
    \end{enumerate}
    All of these subcases satisfy the Condorcet condition. 
    
    Now assume $\zeta_1' \neq \zeta_2'$. Since $-\zeta_1' \ge \zeta_2'$ is equivalent to $-\zeta_2' \ge \zeta_1'$, there are only three distinct pairs to consider:
    \begin{enumerate}
        \item  $\{\zeta_1', \zeta_2'\}=\{\alpha, \gamma\}$: Then $-\gamma \ge \alpha$. Since we already have $\alpha \ge -\gamma$, it follows that $\alpha = -\gamma$.
        \item $\{\zeta_1', \zeta_2'\}=\{\alpha, \beta\}$: Then $-\beta \ge \alpha$. This implies $-\beta \ge \alpha \ge -\gamma$, which is equivalent to $\gamma \ge -\alpha \ge \beta$.
        \item  $\{\zeta_1', \zeta_2'\}=\{\beta, \gamma\}$: Then $-\gamma \ge \beta$. This implies $\alpha \ge -\gamma \ge \beta$.
    \end{enumerate}
    All of these cases also satisfy the Condorcet condition. This completes the proof.
\end{proof}

\subsection{Relation with closed Condorcet domains}\label{subsec: Closed Condorcet domains}
In this subsection, we prove that Condorcet root posets are in bijection with closed Condorcet domains under the maps $\psi$ and $\varphi$. This establishes Theorem~\ref{thm: bijection} and Corollary~\ref{cor:bijection}. 

\begin{lemma}
    \label{lem: psi(P) closed}
    Let $P$ be a Condorcet root poset. Then $\psi(P)$ is a closed Condorcet domain.
\end{lemma}
\begin{proof}
    Let $D = \psi(P)$. Fix an arbitrary $w \in \nu(D)$, we must show that $w \in D$, which is equivalent to $I(w) \in J_s(P)$ from Definition~\ref{def: map psi}. Since $I(w) \sqcup -I(w) = \Phi$, it suffices to show that $I(w)$ is an order ideal in $P$. 
    
    Suppose $\alpha \le_P \beta$ and $\beta \in I(w)$. By Theorem~\ref{thm: fundamental}, there exists a voting profile $f$ which has no tie on $D$ such that $I(f) = I(w)$ and $R(f) = \{w\}$. Since $\beta \in I(w) = I(f)$, we have $f(\beta) > |f|/2$. For any $v \in D$, Definition~\ref{def: map psi} dictates that $I(v) \in J_s(P)$. Since $I(v)$ is an order ideal, $\beta \in I(v)$ implies $\alpha \in I(v)$. Consequently, we have
    \[
        f(\alpha) = \sum_{\substack{v \in D \\ \alpha \in I(v)}} f(v) \ge \sum_{\substack{v \in D \\ \beta \in I(v)}} f(v) = f(\beta) > \frac{|f|}{2}.
    \]
    This implies that $\alpha \in I(f) = I(w)$, which proves $I(w)$ is an order ideal. Thus, $w \in D$, completing the proof.
\end{proof}

\begin{prop}
    \label{prop: phi-psi is closure img psi is closed}
    For any Condorcet domain $D$, $\psi(\varphi(D)) = \overline{D}$. Furthermore, $\psi(\mathcal{P})$ is precisely the set of all closed Condorcet domains. 
\end{prop}
\begin{proof}
Let $P = \varphi(D)$. First, we prove that $\psi(P) = \overline{D}$. By Lemma~\ref{lem: inversions are symmetric order ideals}, Definition~\ref{def: map psi}, and Lemma~\ref{lem: psi(P) closed}, $\psi(P)$ is a closed Condorcet domain containing $D$. Hence, $\psi(P) \supseteq \overline{D}$. 
    
To establish the reverse inclusion $\psi(P) \subseteq \overline{D}$, by Definition~\ref{def: map psi}, we must show that for any $I \in J_s(P)$, there exists $u \in \overline{D}$ such that $I(u) = I$. For any $I \in J_p(P)$, choose an antichain $A$ which generates $I$. Since $J_s(P) \subseteq J_p(P)$, it suffices to prove that for any $I \in J_p(P)$, there exists $u \in \overline{D}$ satisfying $I(u) \supseteq I$. As $u\in \overline{D}\subseteq \psi(P)$, $I(u)\in J_s(P)\subseteq J(P)$. So it's equivalently to prove $I(u)\supseteq A$. We proceed by induction on $|A|$.

When $|A| = 1$, let $A = \{\alpha\}$. The condition $I \in J_p(P)$ implies $\alpha \ngeq_P -\alpha$. By Definition~\ref{def: map phi}, there exists $u \in D$ such that $\alpha \in I(u)$ and $-\alpha \notin I(u)$, which yields the desired result. And when $|A| = 2$, let $A = \{\alpha, \beta\}$. The condition $I \in J_p(P)$ implies $\alpha \ngeq_P -\beta$. Similarly, there exists $u \in D$ such that $\alpha \in I(u)$ and $-\beta \notin I(u)$. The latter implies $\beta \in I(u)$, so $A \subseteq I(u)$, completing this case.
    
For the inductive step, assume $k \ge 2$ and that the claim holds for $|A| = k$. We will show it holds for $|A| = k + 1$. Let $A = \{\alpha_1, \alpha_2, \ldots, \alpha_{k+1}\}$. For each $1 \le i \le k+1$, define $A_i = A \setminus \{\alpha_i\}$ and $I_i$ be the order ideal generated by $A_i$. By the induction hypothesis, there exists $u_i \in \overline{D}$ such that $I_i \subseteq I(u_i)$ for all $i$. 
    
    Consider a voting profile $f$ on $\overline{D}$ such that $f(u_i) = 1$ for all $i$, and $f(v) = 0$ elsewhere. The total number of voters is $|f| = k+1$. For each $\alpha_i$, it appears in at least $k$ of the ideals $I(u_j)$ (for all $j \neq i$), so $f(\alpha_i) \ge k > |f|/2$. Thus, $-\alpha_i\notin I(f)$ for all $i$. By the proof of Theorem~\ref{thm: fundamental}, there exists $u \in R(f)\cap \nu(\overline{D})\subseteq \overline{D}$ such that $-\alpha_i \notin I(u)$ hence $\alpha_i \in I(u)$ for all $i$. Therefore, $A \subseteq I(u)$. This completes the induction.
    
    Finally, the assertion that $\psi(\mathcal{P})$ is exactly the set of all closed Condorcet domains follows directly from Lemma~\ref{lem: psi(P) closed} and the equality $\psi(\varphi(D)) = \overline{D}$ established above.
\end{proof}
\begin{ex}
    \label{ex: phi-psi is closure}
    Consider the Condorcet domain $D=\{213,231,321\}$, which is proved to be closed in Example~\ref{ex: Condorcet and voting} (3) and $\varphi(D)$ was computed in Example~\ref{ex:map-varphi-small-example}. We can see from Figure~\ref{fig:map-varphi-small-example} that $J_s(\varphi(D))=\{\{12,13,23\}, \{23,13,-23\}, \{12,-13,-23\}\}$, hence $\psi(\varphi(D))=\overline{D}=D$.
\end{ex}
\begin{proof}[Proof of Theorem~\ref{thm: bijection} and Corollary~\ref{cor:bijection}]
    By Lemma~\ref{lem:phi well-defined}, Theorem~\ref{thm: Condorcet root poset eq condition}, and Proposition~\ref{prop: phi-psi is closure img psi is closed}, we have $\varphi(\mathcal{D}) = \mathcal{P}$, and $\psi(\mathcal{P})$ is precisely the set of closed Condorcet domains. Moreover, if $P \in \mathcal{P}$ and $D \in \psi(\mathcal{P})$, then $\psi(\varphi(D)) = D$ and $\varphi(\psi(P)) = P$. Therefore, $\psi$ and $\varphi$ restrict to mutually inverse bijections between the set of Condorcet root posets and the set of closed Condorcet domains. Furthermore, since $\psi(\varphi(D)) = \overline{D}$ for any $D\in \mathcal{D}$ and $\psi$ is injective on $\mathcal{P}$, it immediately follows that $\varphi(D_1) = \varphi(D_2)$ if and only if $\overline{D}_1 = \overline{D}_2$.
\end{proof}
\section{Median graphs and Condorcet root posets}
\label{sec: median graph}

Let $G$ be a simple graph. For any two vertices $u, v \in G$, the \emph{distance} $d(u,v)$ is defined as the number of edges in a shortest path connecting $u$ and $v$. A graph $G$ is called a \emph{median graph} \cite[Section 3.1]{Puppe-Slinko-median-graph} if for any three vertices $a, b, c \in G$, there exists a unique vertex $v \in G$ that lies on a shortest path between each pair of $a, b$, and $c$. Equivalently, there is a unique $v \in G$ such that
\[
    d(a,b) = d(a,v) + d(v,b), \quad d(b,c) = d(b,v) + d(v,c), \quad \text{and} \quad d(c,a) = d(c,v) + d(v,a).
\]

Recall that a lattice $L$ is \emph{distributive} if for all $a, b, c \in L$, $a \wedge (b \vee c) = (a \wedge b) \vee (a \wedge c)$. A join-semilattice $L'$ is called a \emph{distributive join-semilattice} if for any $a \in L'$, the principal filter generated by $a$ is a distributive lattice. For instance, for any poset $P$, its lattice of order ideals $J(P)$ is distributive, and any order filter of $J(P)$ is a distributive join-semilattice. 

The following theorem characterizes median graphs in terms of distributive join-semilattices. 

\begin{theorem}[{\cite[Section 5.4]{klavzar1999median}}]
    \label{thm: median graph}
    A graph $G$ is a median graph if and only if it is isomorphic to the Hasse diagram of a join-semilattice $L$ satisfying the following two properties:
    \begin{itemize}
        \item[(1)] $L$ is distributive;
        \item[(2)] For all $x, y, z \in L$, if the pairwise meets $x \wedge y$, $y \wedge z$, and $z \wedge x$ exist, then the meet $x \wedge y \wedge z$ exists.
    \end{itemize}
\end{theorem}

Let $P$ be an antipodal root poset and let $I \in J_s(P)$. We define $J_s(I) := \{I' \cap I \mid I' \in J_s(P)\}$, which we call the set of \emph{symmetric order ideals in $I$}. Let $J(I)$ denote the lattice of order ideals of $I$. Clearly, $J_s(I) \subseteq J(I)$.

\begin{lemma}
    \label{lem: J_s(I)andJ_s(P)}
    Let $P$ be an antipodal root poset and $I \in J_s(P)$. The map $\phi: J_s(P) \to J_s(I)$ given by $I' \mapsto I' \cap I$ is a bijection, with its inverse given by $I_1 \mapsto I_1 \sqcup -(I \setminus I_1)$. Consequently, $|J_s(P)| = |J_s(I)|$.
\end{lemma}
\begin{proof}
    Surjectivity follows immediately from the definition. To show injectivity, let $I_1 \in J_s(I)$ and suppose $I_2 \in J_s(P)$ satisfies $I_2 \cap I = I_1$. Since $I \sqcup -I = I_2 \sqcup -I_2 = P$, we have $I_2 \setminus I = I_2 \cap -I = -(-I_2 \cap I) = -(I \setminus I_2)$. Substituting $I_2 \cap I = I_1$ yields $I_2 \setminus I = -(I \setminus I_1)$. Therefore, $I_2 = (I_2 \cap I) \sqcup (I_2 \setminus I) = I_1 \sqcup -(I \setminus I_1)$. This uniquely determines the inverse map, completing the proof.
\end{proof}

Since $J_s(P)$ is non-empty by Proposition~\ref{prop: J_p(P) in J_s(P)}, such an ideal $I$ always exists. The bijection established above allows us to work primarily with $J_s(I)$ rather than $J_s(P)$. 

For a closed Condorcet domain $D$ and $I\in J_s(\varphi(D))$, we define $\rho: D\rightarrow J_s(I)$ by $\rho(u)=I(u)\cap I$ for any $u \in D$. By Corollary~\ref{cor:bijection}, Definition~\ref{def: map psi} and Lemma~\ref{lem: J_s(I)andJ_s(P)}, $\rho$ is a well-defined bijection. Using $\rho$, we can label $J_s(I)$ by $D$.
\begin{ex}
    \label{ex: J_s(I) and map rho}
    Take the Condorcet root poset $P$ in Figure~\ref{fig:map-psi-example} and an order ideal $I=\{12,13,14,23,24,34\}\in J_s(P)$. We compute the following data:
    \[
    \begin{tabular}{ccc}\toprule
        $v$ & $I(v)$ & $\rho(v)$ \\\midrule
        $4321$ & $I$ & $I$ \\
        $3421$ & $\{12,13,14,23,24,-34\}$ & $I\setminus \{34\}$ \\
        3412 & $\{-12,13,14,23,24,-34\}$ & $I\setminus \{12,34\}$
    \\ \bottomrule\end{tabular}\]
\end{ex}
\begin{prop}
    \label{prop: J_s(I) struture}
    If $I_1 \in J_s(I)$ and $I_1 \subseteq I_2 \in J(I)$, then $I_2 \in J_s(I)$. In other words, $J_s(I)$ is an order filter of $J(I)$.
\end{prop}
\begin{proof}
    For $i \in \{1,2\}$, let $I_i' = I_i \sqcup -(I \setminus I_i) \subseteq P$. By Lemma~\ref{lem: J_s(I)andJ_s(P)}, $I_1' \in J_s(P)$. We claim that $I_2' \in J_s(P)$ as well. By construction, $I'_2 \sqcup -I_2' = P$, so it suffices to show that $I_2'$ is an order ideal of $P$. 
    
    Let $\alpha \in I_2'$ and $\beta \le \alpha$. We must show $\beta \in I_2'$. If $\alpha \in I_2$, then $\beta \in I_2$ because $I_2$ is an order ideal, and thus $\beta \in I_2'$. Suppose instead that $\alpha \in -(I \setminus I_2)$. Since $I_1 \subseteq I_2$, we have $-(I \setminus I_2) \subseteq -(I \setminus I_1) \subseteq I_1'$. Thus, $\alpha \in I_1'$. Because $I_1'$ is an order ideal, this implies $\beta \in I_1'$, which means either $\beta \in I_1$ or $\beta \in -(I \setminus I_1)$. 
    If $\beta \in I_1$, then $\beta \in I_2 \subseteq I_2'$, and the claim holds. Finally, if $\beta \in -(I \setminus I_1)$, then both $-\alpha$ and $-\beta$ belong to $I$. Note that $\beta \le \alpha$ implies $-\alpha \le -\beta$. Since $I_2$ is an order ideal of $I$, $I \setminus I_2$ is an order filter of $I$. As $-\alpha \in I \setminus I_2$, it follows that $-\beta \in I \setminus I_2$, which is equivalent to $\beta \in -(I \setminus I_2) \subseteq I_2'$.
\end{proof}

With these structures, we can now relate closed Condorcet domains to median graphs.
\begin{prop}
    \label{prop: Hasse is median graph}
    Let $D$ be a closed Condorcet domain and $P = \varphi(D)$. For any $I \in J_s(P)$, the Hasse diagram of $J_s(I)$ is a median graph.
\end{prop}
\begin{proof}
    To prove that the Hasse diagram of $J_s(I)$ is indeed a median graph, we verify the two conditions outlined in Theorem~\ref{thm: median graph}. 
    
    For condition (1), by Proposition~\ref{prop: J_s(I) struture}, $J_s(I)$ is an order filter of $J(I)$. Since any order filter of a distributive lattice is a distributive join-semilattice, the condition holds.
    
    For condition (2), let $I_1, I_2, I_3 \in J_s(I)$ such that the pairwise intersections $I_1 \cap I_2$, $I_2 \cap I_3$, and $I_3 \cap I_1$ all belong to $J_s(I)$. We must show that $I_1 \cap I_2 \cap I_3 \in J_s(I)$. 
    Let $u_1, u_2, u_3 \in D$ be the unique elements satisfying $\rho(u_1) = I_1 \cap I_2$, $\rho(u_2) = I_2 \cap I_3$, and $\rho(u_3) = I_3 \cap I_1$. Consider a voting profile $f$ on $D$ such that $f(u_i) = 1$ for $i \in \{1,2,3\}$, and $f(v) = 0$ otherwise. Since the total number of voters is $|f| = 3$, there are no ties. Because $D$ is a closed Condorcet domain, Theorem~\ref{thm: fundamental} guarantees that $I(u)=I(f)$ and the result set is a singleton $R(f) = \{u\} \subseteq D$. 
    
    For any element $\alpha \in I$, $\alpha \in I(f)=\{\alpha \in \Phi\mid f(\alpha)\ge |f|/2=3/2\}$ if and only if $\alpha$ belongs to at least two of the sets $\{\rho(u_1), \rho(u_2), \rho(u_3)\}$. Notice that $\rho(u_1) \cap \rho(u_2) = (I_1 \cap I_2) \cap (I_2 \cap I_3) = I_1 \cap I_2 \cap I_3$, and similar equation holds for $\rho(u_2)\cap \rho(u_3), \rho(u_3)\cap \rho(u_1)$. Therefore, $\alpha\in I(f)$ if and only if $\alpha \in I_1 \cap I_2 \cap I_3$. This implies $\rho(u) = I_1 \cap I_2 \cap I_3$. Since $\rho$ is a bijection onto $J_s(I)$, it follows that $I_1 \cap I_2 \cap I_3 \in J_s(I)$, completing the proof.
\end{proof}
\begin{ex}\label{ex:Hasse is median graph}
Consider the Condorcet root poset $P$ in Figure~\ref{fig:map-psi-example} and take $I=\{12,13,14,23,24,34\}$ as an order ideal in $J_s(P)$. We then compute the poset $J_s(I)$ as shown in Figure~\ref{fig:JsI-example}.
\begin{figure}[h!]
\centering
\begin{tikzpicture}[
node distance=0.6 and 1.2,
every node/.style={draw=none}
]
\node (I) {$I$};
\node (A) [below left=of I] {$I\setminus\{34\}$};
\node (B) [below=of I] {$I\setminus\{12\}$};
\node (C) [below right=of I] {$I\setminus\{23\}$};
\node (D) [below=of A] {$I\setminus\{24,34\}$};
\node (E) [below=of B] {$I\setminus\{12,34\}$};
\node (F) [below=of C] {$I\setminus\{12,13\}$};
\node (G) [below=of E] {$I\setminus\{12,14,34\}$};

\draw (I) -- (A);
\draw (I) -- (B);
\draw (I) -- (C);
\draw (A) -- (D);
\draw (A) -- (E);
\draw (B) -- (E);
\draw (B) -- (F);
\draw (E) -- (G);
\end{tikzpicture}
\caption{The poset $J_s(I)$ in Example~\ref{ex:Hasse is median graph}, with the underlying Condorcet root poset $P$ in Figure~\ref{fig:map-psi-example}}
\label{fig:JsI-example}
\end{figure}
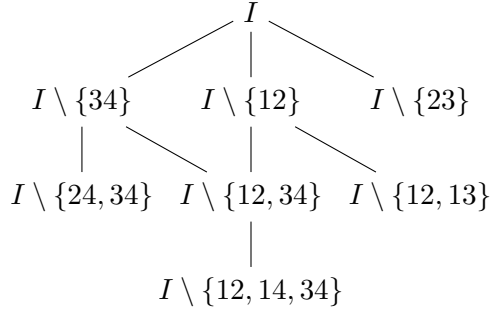

    We can see $J_s(I)$ is a distributive join-semilattice, such as the order filter generated by $I\setminus \{12,14,34\}$ is a distributive lattice. And the Hasse diagram of $J_s(I)$ is a median graph, for example we take $a=I\setminus \{24,34\}, b=I\setminus\{12,13\}, c=I\setminus\{13,14,34\}$, then $v=I\setminus\{12,34\}$ is the unique vertex lies on a shortest path between each pair of $a, b$, and $c$. 
\end{ex}
To ensure that the median graph structure on $D$ is uniquely defined up to isomorphism, we need a characterization of the edges that is independent of the choice of $I\in J_s(\varphi(D))$. The following concept was initially proposed for the type $A$ Condorcet domains in \cite[Section 3.2]{Puppe-Slinko-median-graph}, and achieves our goal.

Let $D$ be a Condorcet domain. Two distinct elements $v_1, v_2 \in D$ are called \emph{neighbors} if there is no $v \in D \setminus \{v_1, v_2\}$ such that $I(v) \supseteq I(v_1) \cap I(v_2)$. Intuitively, this means that the difference between $v_1$ and $v_2$ is minimal. More precisely, $I(v_1) \cap I(v_2)$ acts as a coatom in both $J_s(I(v_1))$ and $J_s(I(v_2))$.

\begin{defin}[{\cite[Section 3.2]{Puppe-Slinko-median-graph}}]
    \label{def:the median graph of D}
    Let $D$ be a closed Condorcet domain. The \emph{associated graph of $D$}, denoted by $G_D$, is the graph with vertex set $D$ and edge set $E(G_D) = \{\{v_1, v_2\} \mid v_1, v_2 \text{ are neighbors}\}$. 
\end{defin}

\begin{prop}
    \label{prop: associated graph is Hasse}
    Let $D$ be a closed Condorcet domain. For any $v \in D$, the map $\rho$ induces a graph isomorphism between $G_D$ and the Hasse diagram of $J_s(I(v))$.
\end{prop}

\begin{proof}
    Fix an element $v \in D$, let $I = I(v)$, and let $T$ denote the Hasse diagram of $J_s(I)$. By Proposition~\ref{prop: Hasse is median graph}, $T$ is a median graph. To show that $G_D \cong T$, we must prove that two distinct elements $v_1, v_2 \in D$ are neighbors if and only if their images $I_1 := \rho(v_1)$ and $I_2 := \rho(v_2)$ satisfy a covering relation $I_1 \lessdot I_2$ or $I_1 \gtrdot I_2$ in $J_s(I)$.

    $(\Rightarrow)$ Suppose $v_1$ and $v_2$ are neighbors. Consider a voting profile $f$ on $D$ such that $f(v) = f(v_1) = f(v_2) = 1$, and $f \equiv 0$ elsewhere. Since the total number of voters is $|f| = 3$, there are no ties. By Theorem~\ref{thm: fundamental} and the assumption that $D$ is closed, there exists a unique $v_3 \in D$ such that $I(v_3) = I(f)$. For any $\alpha \in I(v_1) \cap I(v_2)$, $f(\alpha) \ge 2 > |f|/2$. Thus, $I(v_1) \cap I(v_2) \subseteq I(v_3)$. Because $v_1$ and $v_2$ are neighbors, this forces $v_3 = v_1$ or $v_3 = v_2$. 
    
    Without loss of generality, assume $v_3 = v_1$. Similarly, $I(v_2) \cap I(v)\subseteq I(v_3)=I(v_1)$, which implies $I_1 \ge I_2$ in $J_s(I)$. To show that  $I_1 \gtrdot I_2$, let $v_4 \in D$ be an element such that $I_1 \ge I_4 := \rho(v_4) \ge I_2$. 
    
    By Lemma~\ref{lem: J_s(I)andJ_s(P)}, we have $I(v_i) = I_i \sqcup -(I \setminus I_i)$ for $i \in \{1, 2, 4\}$. Because $I_4 \supseteq I_2 = I_1 \cap I_2$ and $-(I \setminus I_4) \supseteq -(I \setminus I_1) = -(I \setminus I_1) \cap -(I \setminus I_2)$, we deduce that $I(v_4) \supseteq I(v_1) \cap I(v_2)$. Since $v_1$ and $v_2$ are neighbors, we must have $v_4 = v_1$ or $v_4 = v_2$. Therefore, $I_4 = I_1$ or $I_4 = I_2$, meaning that $I_1 \gtrdot I_2$.

    $(\Leftarrow)$ Conversely, suppose $I_1 \lessdot I_2$ or $I_1 \gtrdot I_2$ in $J_s(I)$. Without loss of generality, assume $I_1 \lessdot I_2$. Let $v_5 \in D$ be any element satisfying $I(v_5) \supseteq I(v_1) \cap I(v_2)$. Similar to the logic in the previous paragraph, let $I_5 := \rho(v_5)$, then $I(v_5) \supseteq I(v_1) \cap I(v_2)$ yields $I_5 \supseteq I_1$ and $-(I \setminus I_5) \supseteq -(I \setminus I_2)$. This implies $I_1 \le I_5 \le I_2$ in $J_s(I)$. Because $I_2$ covers $I_1$, it must be that $I_5 = I_1$ or $I_5 = I_2$. This means $v_5 = v_1$ or $v_5 = v_2$, which proves that $v_1$ and $v_2$ are indeed neighbors.
\end{proof}

Now we are ready to prove that every closed Condorcet domain is associated to a unique median graph.

\begin{proof}[Proof of Theorem~\ref{thm: Condorcet domain to median graph}]
   The result follows directly from the combination of Proposition~\ref{prop: associated graph is Hasse}, Proposition~\ref{prop: Hasse is median graph}.
\end{proof}
\begin{ex}
    \label{ex:associated graph is Hasse}
    Take the Condorcet root poset $P$ in Figure~\ref{fig:map-psi-example}, as partially showed in Example~\ref{ex:map-psi-example}, $D=\psi(P)=\{3142,3241,3412,3421,4132,4231,4312,4321\}$. Then $u=4321$ and $v=3421$ are neighbors as $I(u)\cap I(v)=\{12,13,14,23,24\}$ has size $5$, the only possibility of $w\in D$ with $I(w)\supseteq I(u)\cap I(v)$ is $I(w)=I(u)$ or $I(w)=I(v)$, hence $w=u$ or $w=v$. We can compute the associated graph $G_D$ of $D$ shown in Figure~\ref{fig:GD-example}.
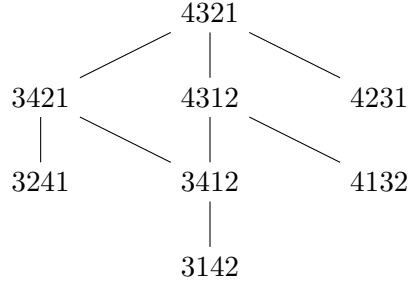
\begin{figure}[h!]
\centering
\begin{tikzpicture}[
node distance=0.6 and 1.2,
every node/.style={draw=none}
]
\node (I) {$4321$};
\node (A) [below left=of I] {$3421$};
\node (B) [below=of I] {$4312$};
\node (C) [below right=of I] {$4231$};
\node (D) [below=of A] {$3241$};
\node (E) [below=of B] {$3412$};
\node (F) [below=of C] {$4132$};
\node (G) [below=of E] {$3142$};

\draw (I) -- (A);
\draw (I) -- (B);
\draw (I) -- (C);
\draw (A) -- (D);
\draw (A) -- (E);
\draw (B) -- (E);
 \draw (B) -- (F);
\draw (E) -- (G);
\end{tikzpicture}
\caption{The associated graph $G_D$ in Example~\ref{ex:associated graph is Hasse}, which is isomorphic with the graph in Figure~\ref{fig:JsI-example}}
\label{fig:GD-example}
\end{figure}

Combined with Example~\ref{ex:Hasse is median graph}, we can see the Hasse diagram of $J_s(I)$ is isomorphic with $G_D$, and then $G_D$ is a median graph.
\end{ex}

In light of Theorem~\ref{thm: Condorcet domain to median graph}, the associated graph $G_D$ is canonically referred to as the \emph{median graph of $D$}. The converse relationship was established by Puppe and Slinko in \cite[Theorem 5]{Puppe-Slinko-median-graph} for type $A$ domains as follows.

\begin{theorem}[{\cite[Theorem 5]{Puppe-Slinko-median-graph}}]
    \label{thm: median graph to Condorcet domain}
    For any median graph $G$, there exists a Condorcet domain $D$ (of type $A$) whose median graph is isomorphic to $G$.
\end{theorem}

As observed in \cite{Puppe-Slinko-median-graph}, the Condorcet domain $D$ corresponding to a given median graph in Theorem~\ref{thm: median graph to Condorcet domain} is not unique. This naturally raises the question of how to characterize the equivalence classes of Condorcet domains that yield the same median graph. We use the notion of the \emph{skeleton isomorphism classes} of Condorcet root posets to answer this question, as formulated in Theorem~\ref{thm: eq median graph iff eq skeleton}.

\begin{defin}
    \label{def: skeleton isomorphic}
    Let $P$ and $P'$ be preposets. We say a map $\chi: P \rightarrow P'$ is a \emph{order equivalence} if it is an order isomorphism after modding out the equal relations in the preposets, or equivalently, if it is an order embedding, that is, $a \le_P b$ if and only if $\chi(a) \le_{P'} \chi(b)$ for all $a,b \in P$, and it is essentially surjective, meaning that for every $y \in P'$, there exists an $x \in P$ such that $\chi(x) =_{P'} y$.
    Two Condorcet root posets $P_1$ and $P_2$ are said to be \emph{skeleton isomorphic} if there exists an order equivalence $\chi: \tilde{P}_1 \to \tilde{P}_2$ such that $\chi(-\alpha) = -\chi(\alpha)$ for all $\alpha \in \tilde{P}_1$. 
\end{defin}

Recall that in a lattice $L$, an element $a \in L$ is \emph{meet-irreducible} if $a$ is not the maximum element in $L$, and $a = b \wedge c$ implies $a = b$ or $a = c$. A classic theorem by Birkhoff allows us to recover a poset from a distributive lattice using its meet-irreducible elements.

\begin{theorem}[{\cite[Section 3.4]{stanleyEnumerativeCombinatoricsVolume2012}}]
    \label{thm: Birkhoff's representation theorem}
    For any distributive lattice $L$, let $P$ be the subposet of meet-irreducible elements in $L$. Then $L$ is order-isomorphic to $J(P)$ via the map $a \mapsto \{b \in P \mid b \le a\}$.
\end{theorem}

The set $J_s(I)$ behaves very well with respect to meet-irreducible elements.

\begin{lemma}
    \label{lem: meet-irreducible}
    Let $P$ be a Condorcet root poset, and let $I \in J_s(P)$. Suppose $I' \in J(I)$ satisfies the condition that $I' \neq \emptyset$ whenever $P_m \neq \emptyset$. Then $I'$ is a meet-irreducible element in $J(I)$ if and only if $I'$ is a meet-irreducible element in $J_s(I)$. 
\end{lemma}
\begin{proof}
    Whether $I'$ is meet-irreducible depends only on the structure of $[I', I]$. By Proposition~\ref{prop: J_s(I) struture}, if $I'\in J_s(I)$, then $[I', I] \subseteq J_s(I) \subseteq J(I)$. Therefore $I'$ is meet-irreducible in $J(I)$ if and only if it is meet-irreducible in $J_s(I)$. We only need to prove the converse: if $I'$ is meet-irreducible in $J(I)$, then $I' \in J_s(I)$. 
    
    Let $A \subseteq I$ be the antichain that generates the order filter $I \setminus I'$ in $I$. If $A = A_1 \sqcup A_2$ for some nonempty disjoint sets $A_1$ and $A_2$, let $F_1$ and $F_2$ be the order filters generated by $A_1$ and $A_2$ in $I$, respectively. Then $I' = (I \setminus F_1) \cap (I \setminus F_2)$. Since $I \setminus F_i \neq I'$ for $i \in \{1,2\}$, this contradicts the assumption that $I'$ is meet-irreducible. 
    
    Furthermore, when $P_m \neq \emptyset$, we are given that $I' \neq \emptyset$. If $A = P_m$, then $I \setminus I'$ would be exactly the filter generated by $P_m$, which is $\emptyset$, yielding a contradiction. Thus, $A \neq P_m$. Because $I'$ is not the maximum element $I$, $A$ cannot be empty. Therefore, $A$ must be a singleton which does not contained in $P_m$, say $A = \{\alpha\}$ for some $\alpha \notin P_m$. 
    
    To prove that $I' \in J_s(I)$, motivated by Lemma~\ref{lem: J_s(I)andJ_s(P)}, we define $I'' := I' \sqcup -(I \setminus I')$. We must prove $I'' \in J_s(P)$. By construction, $I'' \sqcup -I'' = P$, so it suffices to prove that $I''$ is an order ideal of $P$. 
    
    Let $\beta \in I''$ and $\gamma \le \beta$. We must show that $\gamma \in I''$. 
    If $\beta \in I'$, then since $I'$ is an order ideal of $I$ and $I$ is an order ideal of $P$, $\gamma \in I'$, which implies $\gamma \in I''$. 
    Suppose instead that $\beta \in -(I \setminus I')$. This means $-\beta \in I \setminus I'$. Because $I \setminus I'$ is the order filter generated by $\alpha$ in $I$, we have $-\beta \ge \alpha$. The inequality $\gamma \le \beta$ implies $-\gamma \ge -\beta$, and thus $-\gamma \ge \alpha$. 
    If $-\gamma \in I$. Since $-\gamma \ge \alpha$, we have $-\gamma \in I \setminus I'$. This immediately implies $\gamma \in -(I \setminus I') \subseteq I''$.
        
    If $-\gamma \notin I$, which means $\gamma \in I$. Assume for the sake of contradiction that $\gamma \notin I''$. Because $I'' \cap I = I'$, we have $\gamma \in I \setminus I'$, which means $\gamma \ge \alpha$. Then we have $-\alpha \ge \beta \ge \gamma \ge \alpha$.
    This implies $\alpha \le -\alpha$, contradicts $\alpha \notin P_m$. Thus $\gamma \in I''$. We are done.
\end{proof}

\begin{proof}[Proof of Theorem~\ref{thm: eq median graph iff eq skeleton}]
    Let $G = G_D$, $G' = G_{D'}$, $P = \varphi(D)$, and $P' = \varphi(D')$.
    
    $(\Leftarrow)$ Suppose $P$ and $P'$ are skeleton isomorphic, and let $\chi: \tilde{P} \to \tilde{P}'$ be the skeleton isomorphism. For any order ideal $I \in J_p(P)$  with $P_m\subseteq I$, define $\chi'(I)$ to be the order ideal generated by $\chi(I \setminus P_m)\sqcup P_m'$ in $P'$. By Definition~\ref{def: skeleton isomorphic}, $\chi'$ induce a bijection from $J_s(P)$ to $J_s(P')$. Take any $v \in D$. By Lemma~\ref{lem: inversions are symmetric order ideals}, $I(v) \in J_s(P)$, and thus $\chi'(I(v)) \in J_s(P')$. Corollary~\ref{cor:bijection} and Definition~\ref{def: map psi} then guarantees the existence of a unique $v' \in D'$ such that $I(v') = \chi(I(v))$. We define $\chi_0(v) := v'$, which makes $\chi_0$ a bijection from $D$ to $D'$. 
    
    For any $v_1, v_2, v_3 \in D$, we have $I(v_3) \supseteq I(v_1) \cap I(v_2)$ if and only if 
    \[
        I(v'_3) = \chi'(I(v_3)) \supseteq \chi'(I(v_1) \cap I(v_2)) = I(v_1') \cap I(v_2').
    \]
    Thus, $v_1$ and $v_2$ are neighbors in $D$ if and only if $\chi_0(v_1)$ and $\chi_0(v_2)$ are neighbors in $D'$. This proves that $\chi_0$ is a graph isomorphism between $G$ and $G'$. 

    $(\Rightarrow)$ Suppose $G$ and $G'$ are isomorphic, and let $\chi_0: G \to G'$ be the isomorphism. Fix any $v \in G$ and let $v' = \chi_0(v) \in G'$. Let $I = I(v)$ and $I' = I(v')$. By Proposition~\ref{prop: associated graph is Hasse}, the map $\rho \circ \chi_0 \circ \rho^{-1}$ induces an isomorphism between the Hasse diagram of $J_s(I)$ and the Hasse diagram of $J_s(I')$ which preserves the maximum element, that is, maps $I$ to $I'$. Since Hasse diagram isomorphisms which preserve the maximum element correspond to order isomorphisms, $\rho \circ \chi_0 \circ \rho^{-1}$ induces an order isomorphism $\chi_1: J_s(I) \to J_s(I')$.  Then by the bijection in Lemma~\ref{lem: J_s(I)andJ_s(P)}, $\chi_1$ induces a bijection $\chi_2: J_s(P) \to J_s(P')$. 
    
    Let $I_0$ and $I'_0$ be the subposets of meet-irreducible elements in $J_s(I)$ and $J_s(I')$, respectively. By Theorem~\ref{thm: Birkhoff's representation theorem} and the proof of Lemma~\ref{lem: meet-irreducible}, $I_0$ is order equivalent to $I \setminus P_m$, and $I'_0$ is order equivalent to $I' \setminus P'_m$. Thus, the restriction $\chi_1|_{I_0}$ yields an order equivalence $\chi: I \setminus P_m \to I' \setminus P'_m$. 
    
    By Theorem~\ref{thm: Birkhoff's representation theorem}, for any $I_1 \in J_s(I)$, $\chi_1(I_1)$ can be explicitly formulated as the order ideal generated by $\chi(I_1 \setminus P_m)\sqcup P'_m$ in $P'$. So for all $\alpha \in I\setminus P_m$, $\alpha \in I_1$ if and only if $\chi(\alpha)\in \chi_1(I_1)$, which implies $\chi_2(J_\alpha(P)) = J_{\chi(\alpha)}(P')$.
    
    Note we have the following equations:
    \begin{align*}
        &J_\alpha(P)\sqcup J_{-\alpha}(P)  = J_s(P), &  &\tilde{P} = (I \setminus P_m) \sqcup -(I \setminus P_m),\\
        &J_{\chi(\alpha)}(P')\sqcup J_{-\chi(\alpha)}(P') = J_s(P'), & &\tilde{P}' = (I' \setminus P'_m) \sqcup -(I' \setminus P'_m).
    \end{align*}
   We define $\chi(-\alpha):=-\chi(\alpha)$ for $\alpha \in I\setminus P_m$. Therefore, $\chi_2$ is a bijection from $J_s(P)$ to $J_s(P')$ with $\chi_2(J_\alpha(P)) = J_{\chi(\alpha)}(P')$ for all $\alpha \in \tilde{P}$, which makes $\chi$ a skeleton isomorphism from $\tilde{P}$ to $\tilde{P}'$ by Proposition~\ref{prop: recover from symmetric order ideal}.
\end{proof}
\section{Families of Condorcet domains}\label{sec: families}
In this section, we provide several equivalent descriptions of certain useful families of Condorcet domains in terms of their corresponding Condorcet root posets. Using these equivalent descriptions, we are able to generalize and strengthen various properties of these families. The classical type $A$ definitions and properties can be found in \cite{puppe2024maximal, karpovConstructingLargePeakpit2023}.

Here are some notations that will be used throughout this section. For any two posets $P_1, P_2$ on the same ground set $A$, we say that $P_2$ \emph{refines} $P_1$, denoted by $P_1\subseteq P_2$, if $a \le_{P_1} b $ implies $a \le_{P_2} b$ for $a, b \in A$. Similarly, define the \emph{intersection} of $P_1$ and $P_2$, denoted by $P_1\cap P_2$, to be the poset $P_3$ with $a\le_{P_3} b$ if and only if $a \le_{P_1} b $ and $a \le_{P_2} b$ for $a, b \in A$.

\subsection{Maximal, ample, and connected}\label{subsec: Maximal ample and connected}

We say that a Condorcet root poset $P$ is \emph{minimal} if there is no Condorcet root poset $P'$ on the same root system $\Phi$ such that $P'\subsetneq P$. The following lemma connects maximality and minimality between the classes $\mathcal{D}$ and $\mathcal{P}$.

\begin{lemma}
    \label{lem: contravariant}
    The maps $\varphi$ and $\psi$ are contravariant between $\mathcal{D}$ and $\mathcal{P}$, that is, if $D\subseteq D' \in \mathcal{D}$ (resp.\ $P\subseteq P' \in \mathcal{P}$), then $\varphi(D)\supseteq \varphi(D')$ (resp.\ $\psi(P)\supseteq \psi(P')$).
\end{lemma}
\begin{proof}
    Suppose $D\subseteq D'$ and $\alpha \le \beta$ in $\varphi(D')$. Then $D'(\alpha)=\{v\in D' \mid \alpha \in I(v)\}\supseteq D'(\beta)=\{v\in D' \mid \beta \in I(v)\}$. Since $D\subseteq D'$, we have $D(\alpha)=D\cap D'(\alpha)\supseteq D\cap D'(\beta)=D(\beta)$. Thus, $\alpha\le \beta$ in $\varphi(D)$, which implies $\varphi(D)\supseteq \varphi(D')$. 

    Conversely, suppose $P\subseteq P'$ and $v \in \psi(P')$. Then $I(v)\in J_s(P')$. For any $\alpha\in I(v)$ and $\beta\in \Phi$ such that $\beta\le_P \alpha$, the inclusion $P\subseteq P'$ implies $\beta\le_{P'}\alpha$, and therefore $\beta\in I(v)$. Thus, $I(v)\in J_s(P)$, which means $v\in \psi(P)$, yielding $\psi(P)\supseteq \psi(P')$.
\end{proof}

\begin{prop}
    \label{prop: maximal and minimal}
    A closed Condorcet domain $D$ is maximal if and only if $\varphi(D)$ is minimal.
\end{prop}
\begin{proof}
    This follows immediately from Corollary~\ref{cor:bijection} and Lemma~\ref{lem: contravariant}.
\end{proof}

\begin{ex}
    \label{ex: maximal in type I}    
    By Proposition~\ref{prop: maximal and minimal}, Theorem~\ref{thm: cycle condition} immediately yields a classification of maximal type $I$ Condorcet domains, as their root systems are simply dihedral root cycles. The two types of cycle conditions produce two types of maximal Condorcet domains in $I_2(n)$, having sizes $n+1$ and $4$, respectively. For instance, let $\{s_1, s_2\}$ be the set of generators of $I_2(n)$, then $\{e, s_1, s_1s_2, s_1s_2s_1, \ldots, w_0\}$ and $\{e, w_0, s_1, s_1w_0\}$ are maximal Condorcet domains that corresponding to the two types of cycle conditions, respectively.
\end{ex}

We say that a Condorcet domain $D$ is \emph{ample} if there is no $\alpha\in \Phi$ such that $\alpha\in I(v)$ for all $v\in D$.

\begin{prop}
    \label{prop: ample in poset}
    Let $D$ be a closed Condorcet domain and let $P=\varphi(D)$. Then $D$ is ample if and only if $P_G=\emptyset$.
\end{prop}
\begin{proof}
    By Corollary~\ref{cor:bijection}, Definitions~\ref{def: map phi} and ~\ref{def: map psi}, $D$ is ample if and only if there is no $\alpha\in \Phi$ such that $\alpha\in I$ for all $I\in J_s(P)$, which is equivalent to saying that there is no $\alpha \in P_m$. Since $P_G=P_m\sqcup (-P_m)$, the result follows.
\end{proof}

Perhaps surprisingly, the following example demonstrates that a maximal domain is not necessarily ample.
\begin{ex}
    \label{ex: maximal doesn't imply ample}
    Consider the type $A_4$ domain \[ D=\{12345, 12354, 12534, 13245, 13254, 13524, 14235, 14325, 41235, 41325, 42135, 43125\}. \] Let $P=\varphi(D)$, which is is depicted in Figure~\ref{fig: maximal not ample}. One can verify that $D$ is a maximal Condorcet domain with $-15 \in P_m$. 
\begin{figure}[htbp]
    \centering
    \resizebox{0.3\textwidth}{!}{
        \begin{tikzpicture}[
            x=1.5cm, 
            y=1cm, 
            every node/.style={
                rectangle, 
                rounded corners=4pt, 
                draw, 
                fill=white, 
                inner sep=2pt, 
                minimum width=1.1cm, 
                minimum height=0.6cm
            }
        ]

        \node (n_15_b) at (1.5, 0) {-15};

        \node (n_35_b) at (2, 1) {-35};
        \node (n_25_b) at (0, 1) {-25};
        \node (n_13_b) at (1, 1) {-13};
        \node (n_12_b) at (3, 1) {-12};

        \node (n_45_b) at (1, 2) {-45};
        \node (n_14_b) at (2, 2) {-14};

        \node (n_2434)   at (1, 3) {24, 34};
        \node (n_m24m34) at (2, 3) {-24, -34};
        \node (n_23_b) at (0, 3) {-23};
        \node (n_23)   at (3, 3) {23};

        \node (n_14) at (1, 4) {14};
        \node (n_45) at (2, 4) {45};

        \node (n_12) at (0, 5) {12};
        \node (n_13) at (2, 5) {13};
        \node (n_35) at (1, 5) {35};
        \node (n_25) at (3, 5) {25};

        \node (n_15) at (1.5, 6) {15};

        \draw (n_15_b) -- (n_35_b);
        \draw (n_15_b) -- (n_25_b);
        \draw (n_15_b) -- (n_13_b);
        \draw (n_15_b) -- (n_12_b);

        \draw (n_35_b) -- (n_45_b);
        \draw (n_35_b) -- (n_23);
        \draw (n_25_b) -- (n_45_b);
        \draw (n_25_b) -- (n_23_b);
        \draw (n_13_b) -- (n_23_b);
        \draw (n_13_b) -- (n_14_b);
        \draw (n_12_b) -- (n_23);
        \draw (n_12_b) -- (n_14_b);

        \draw (n_45_b) -- (n_2434);
        \draw (n_23_b) -- (n_12);
        \draw (n_23_b) -- (n_35);
        \draw (n_23)   -- (n_13);
        \draw (n_23)   -- (n_25);
        \draw (n_14_b) -- (n_m24m34);

        \draw (n_2434)   -- (n_14);
        \draw (n_m24m34) -- (n_45);

        \draw (n_14) -- (n_12);
        \draw (n_14) -- (n_13);
        \draw (n_45) -- (n_35);
        \draw (n_45) -- (n_25);

        \draw (n_12) -- (n_15);
        \draw (n_13) -- (n_15);
        \draw (n_35) -- (n_15);
        \draw (n_25) -- (n_15);

        \end{tikzpicture}
    }
    \vspace{0.5em}
    \caption{The poset in Example~\ref{ex: maximal doesn't imply ample}, which corresponding to a maximal but not ample Condorcet domain}
    \label{fig: maximal not ample}
\end{figure}
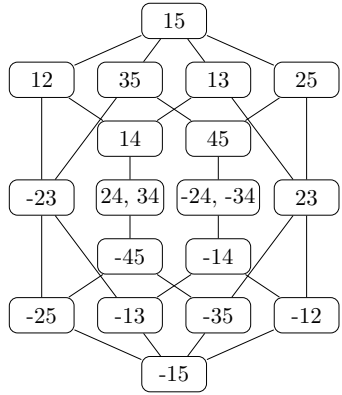
\end{ex}

We say that a Condorcet domain $D$ is \emph{connected} if for any $u, v\in D$, there exist $m\ge0$ and a sequence $(w_0=u, w_1, \ldots, w_m=v)$ in $D$ such that $|I(w_{i+1})\setminus I(w_i)|=1$ for all $0\le i\le m-1$.

\begin{prop}
    \label{prop: connected in poset}
    Let $D$ be a closed Condorcet domain and let $P=\varphi(D)$. Then $D$ is connected if and only if there are no equal roots in $\tilde{P}$, that is, for any $\alpha=\beta \in \tilde{P}$, we have $\alpha=_\Phi \beta$.
\end{prop}
\begin{proof}
    Suppose $D$ is connected and $\alpha=\beta \in \tilde{P}$. Then $\alpha, -\alpha \notin P_G$, and by Definition~\ref{def: map phi}, there exist $I, I'\in J_s(P)$ such that $\alpha\in I$ and $\alpha \notin I'$. By Corollary~\ref{cor:bijection}, we can choose $v, v'\in D$ such that $I(v)=I$ and $I(v')=I'$. As $D(\alpha)=D(\beta)$, by the definition of connectedness on $v, v'$, there exist $u, w\in D$ such that $|I(u)\setminus I(w)|=1$, $\alpha, \beta \notin I(u)$, and $\alpha, \beta \in I(w)$. Then $\alpha=_\Phi \beta$.

    Conversely, suppose that for any $\alpha=\beta \in \tilde{P}$, we have $\alpha=_\Phi \beta$. Take any $u, v \in D$, and let $I=I(v)\in J_s(P)$ and $I'=I(u)\cap I\in J_s(I)$. By Proposition~\ref{prop: J_s(I) struture}, any $I''\in J(I)$ with $I''\supseteq I'$ satisfies $I''\in J_s(I)$. Let $m=|I\setminus I'|$. Then, similarly to the proof of Proposition~\ref{prop: J_p(P) in J_s(P)}, we can construct a sequence $(I_0=I', I_1,\ldots, I_m=I)$ in $J_s(I)$ such that $I_{i+1}=I_i\sqcup \{\alpha_i\}$. Choosing $w_i\in D$ such that $\rho(w_i)=I_i$ for each $i$, the result follows.
\end{proof}

\begin{proof}[Proof of Theorem~\ref{thm: minimal cover is root-triple}]
    $(\Rightarrow)$ Suppose $D$ is maximal. By Propositions~\ref{prop: ample in poset} and ~\ref{prop: connected in poset}, $P=\tilde{P}$ and there are no equal roots in $P$. Let $P'$ be the poset obtained from $P$ by removing the relations $\alpha \le \beta$ and $-\beta \le -\alpha$. Because these are covering relations in $P$ (i.e., $\alpha \lessdot \beta$ and $-\beta \lessdot -\alpha$) and there are no other roots equal $\alpha$ or $\beta$, $P'$ is a well-defined poset. Because $P'\subsetneq P$, Proposition~\ref{prop: maximal and minimal} implies that $P'$ is not a Condorcet root poset. Since $P_G=\emptyset$, $P'$ satisfies conditions (1) and (2) of Definition~\ref{def: condorcet-poset}. Thus, there exists a root-triple $\{\alpha', \beta', \gamma'\}$ that satisfies the Condorcet condition in $P$ but not in $P'$. By the construction of $P'$, we must have either $\alpha, -\beta \in \{\alpha', \beta', \gamma'\}$ or $\alpha, -\beta \in -\{\alpha', \beta', \gamma'\}$. In either case, the claim holds.

    $(\Leftarrow)$ Suppose for any $\alpha, \beta\in P$, the covering relation $\alpha\lessdot \beta$ implies the existence of a root-triple containing $\alpha$ and $-\beta$. By Proposition~\ref{prop: maximal and minimal}, we only need to prove $P$ is minimal. Suppose, for the sake of contradiction, that there exists a Condorcet root poset $P'\subsetneq P$. By Proposition~\ref{prop: ample in poset}, $P=\tilde{P}$, which in turn implies $P'=\tilde{P'}$. By Proposition~\ref{prop: connected in poset}, there are no equal roots in $P$, and therefore none in $P'$. Since $P'\subsetneq P$, there exists a covering relation $\alpha\lessdot \beta$ in $P$ such that $\alpha$ and $\beta$ are incomparable in $P'$. Let $\{\alpha, -\beta, \gamma\}$ be the root-triple guaranteed by the hypothesis. Since there are no equal roots in $P'$, $\{\alpha, -\beta, \gamma\}$ satisfies the first type of Condorcet condition in $P'$. Because $\alpha$ and $\beta$ are incomparable in $P'$, the Condorcet condition implies that either $\alpha\le -\gamma\le -\beta$ or $\alpha \ge -\gamma \ge -\beta$ in $P'$. If $\alpha\le -\gamma\le -\beta$, then $-\alpha \ge \beta \ge \alpha$ in $P$. Conversely, if $\alpha \ge -\gamma \ge -\beta$, then $\beta \ge \alpha\ge -\beta$ in $P$. Both cases contradict the fact that $P_G=\emptyset$. Thus, no such Condorcet root poset $P'\subsetneq P$ can exist, completing the proof.
\end{proof}

Theorem~\ref{thm: minimal cover is root-triple} makes it easy to determine whether a connected and ample Condorcet domain $D$ is maximal: we can simply examine the Hasse diagram of $P$ and verify whether every edge (or more precisely, half of the edges) satisfies the condition in this theorem.

\begin{ex}
    \label{ex: minimal cover is root-triple}
    Let $P$ be the Condorcet root poset in Figure~\ref{fig:map-psi-example}, and let $D=\psi(P)$. One can verify that $D$ is a maximal Condorcet domain, and by Propositions~\ref{prop: ample in poset} and \ref{prop: connected in poset}, we can see directly from $P$ that $D$ is both ample and connected. Therefore, for any covering relation $\alpha \lessdot \beta$ in $P$, the roots $\alpha$ and $-\beta$ belong to a common root-triple. For instance, taking $\alpha=12$ and $\beta=-23$, the set $\{12, 23, -13\}$ forms such a root-triple.
\end{ex}
\subsection{Peak-pit and saturated single-crossing}\label{subsec: Peak-pit and saturated single-crossing}
We say that a Condorcet domain $D$ is \emph{peak-pit} if every root-triple satisfies the first type of Condorcet condition in $\varphi(D)$. 

\begin{prop}
    \label{prop: connected implies peak-pit}
    If a closed Condorcet domain $D$ is connected, then $D$ is peak-pit.
\end{prop}
\begin{proof}
    Suppose $D$ is not peak-pit. Then there exists a root-triple $\{\alpha, \beta, \gamma\}$ such that $\alpha=-\beta$ in $P$ and does not satisfy the first type of Condorcet condition. If $\alpha\in P_G$, we may assume without loss of generality that $\alpha\in P_m$. This implies $\alpha\le -\beta\le \gamma$, which satisfies the first type of Condorcet condition, yielding a contradiction. Thus, we must have $\alpha=-\beta \in \tilde{P}$. However, by Proposition~\ref{prop: connected in poset}, this implies that $D$ is not connected, which is a contradiction. Therefore, $D$ must be peak-pit.
\end{proof}

A nonempty subset $D\subseteq W$ is a \emph{saturated single-crossing domain} if $D=\{v_1,v_2,\ldots,v_m\}$ and there exist distinct roots $\alpha_1,\alpha_2,\ldots,\alpha_{m-1} \in \Phi$ such that $I(v_{i+1})\setminus I(v_{i})= \{\alpha_{i}\}$ for $1\le i\le m-1$. In this case, we say $D$ \emph{connects} $v_1$ and $v_m$. For instance, the Condorcet domain $D$ in Example~\ref{ex:map-varphi-small-example} is a saturated single-crossing domain connecting $213$ and $321$.

We now recall some well-known concepts from the theory of Coxeter groups. The \emph{longest element} in $W$ is denoted by $w_0$, it is the unique element in $W$ for which $I(w_0)=\Phi^+$. For any $w\in W$, we have $I(w)\sqcup I(ww_0)=\Phi$. Let $S=\{s_1, \ldots, s_n\}$ be the set of \emph{simple reflections} generating the Coxeter group $W$. We say that a sequence of generators $\mathbf{i}=(s_{i_1}, s_{i_2}, \ldots, s_{i_m})$ is a \emph{reduced word} for $w\in W$ if $w=s_{i_1}s_{i_2}\cdots s_{i_m}$, where $m$ is minimum possible. We call $m$ the \emph{length} of $w$, denoted by $\ell(w)$. For any reduced word $\mathbf{i}$, any contiguous subsegment of $\mathbf{i}$ is also a reduced word. Let $\{\gamma_k\}:= I(s_k)\cap \Phi^+$ be the corresponding \emph{simple root} of $s_k$ for $1\le k \le n$. Then $w$ has reduced word $\mathbf{i}$ implies $I(w)\cap \Phi^+=\{\gamma_{i_1}, s_{i_1}(\gamma_{i_2}), s_{i_1}s_{i_2}(\gamma_{i_3}),\ldots, s_{i_1}s_{i_2}\cdots s_{i_{m-1}}(\gamma_{i_{m}})\}$. So for any $w\in W$, we have $\ell(w)=|I(w)\cap \Phi^+|$. 

For any $u, v \in W$, we have $I(uv)=uI(v)$ and for any root-triple $T$, the set $uT$ remains a root-triple, it follows that for any $w\in W$, $D$ is a Condorcet domain if and only if $wD$ is a Condorcet domain. Furthermore, $\varphi(wD)$ is simply a relabeling of $\varphi(D)$, which is obtained by replacing any $\alpha \in \Phi$ by $w(\alpha)$. Therefore, we may always assume without loss of generality that $w_0\in D$.

\begin{prop}
    \label{prop: saturated single-crossing domain exists}
    Let $D$ be a saturated single-crossing domain and let $P=\varphi(D)$. Then $D$ is a connected, closed Condorcet domain, and $\tilde{P}$ is a disjoint union of two chains. Furthermore, for any $u, v \in W$, there exists a saturated single-crossing domain connecting $u$ and $v$. Moreover, when $v=w_0$, we can choose the roots $\alpha_i$ from $\Phi^+$.
\end{prop}
\begin{proof}
    Suppose $D$ is not a Condorcet domain. By Definition~\ref{def: condorcet}, there exist a root-triple $\{\alpha, \beta, \gamma\}$ and distinct elements $u, v, w\in D$ such that $\alpha,\beta \in I(u)$, $\beta, \gamma\in I(v)$, and $\gamma, \alpha\in I(w)$. We assume without loss of generality that $u=v_i$, $v=v_j$, and $w=v_k$ with $i<j<k$. Since $\alpha \in I(u)\setminus I(v)$, the construction of $D$ implies that there exists an index $l$ with $i\le l<j$ such that $\alpha=-\alpha_l$. Because $l< k$, we have $\alpha_l=-\alpha\in I(v_k)=I(w)$, contradicting the fact that $\alpha \in I(w)$. Thus, $D$ is Condorcet.

    By definition, $D$ is connected. To show that $D$ is closed, Corollary~\ref{cor:bijection} implies that it suffices to prove $\psi(P)=D$. By the construction of $D$, for any $\alpha \in I(v_1)\cap I(v_m)$, we have $D(\alpha)=D$. Additionally, for $1\le i\le m-1$, $D(\alpha_i)=\{v_{i+1}, v_{i+2},\ldots, v_{m}\}$. Since $(I(v_1)\cap I(v_m)) \cup \{\alpha_1, \alpha_2,\ldots, \alpha_{m-1}\}=I(v_m)$, $I(v_m)\sqcup -I(v_m)=\Phi$, and $D(\alpha)\sqcup D(-\alpha)=D$, we can determine $D(\alpha)$ for all $\alpha\in \Phi$ now. Thus, we find $P_m= I(v_1)\cap I(v_m)$ and $P_M=-P_m=\Phi\setminus (I(v_1)\cup I(v_m))$. Furthermore, the only relations in $\tilde{P}$ form two chains: $\alpha_1<\alpha_2<\cdots <\alpha_{m-1}$ and $-\alpha_{m-1}<-\alpha_{m-2}<\cdots<-\alpha_1$. There are exactly $m$ symmetric order ideals in $P$, which correspond exactly to $v_1, v_2, \ldots, v_m$, hence, $\psi(P)=D$, as required.

    Now, let $u, v\in W$, we will prove there is a saturated single-crossing domain connecting $u$ and $v$. As for any $w_1, w_2\in W$, $I(w_1w_2)=w_1I(w_2)$, multiplying from the left preserves the saturated single-crossing property. By multiplying by $w_0v^{-1}$ from the left, we may assume $v=w_0$. Then there exists a reduced word $\mathbf{i}=(s_{i_1}, s_{i_2},\ldots , s_{i_m})$ for $w_0$ and an integer $l\le m$ such that $\mathbf{i'}=(s_{i_1}, s_{i_2},\ldots , s_{i_l})$ is a reduced word for $u$. Define $D=\{v_j=s_{i_1}s_{i_2}\cdots s_{i_{l+j-1}}\in W \mid  1\le j\le m-l+1\}$, and let $\alpha_j=v_j(\gamma_{i_{l+j}})$. Then $v_1=u, v_{m+l-1}=v$, and we have $I(v_{j+1})\setminus I(v_j)=\{\alpha_j\}\subseteq \Phi^+ $ for $1\le j \le m-l$. So this construction yields the desired domain. 
\end{proof}

\begin{lemma}
    \label{lem: union to intersection}
    For any Condorcet domains $D$ and $D'$, $\varphi(D\cup D')=\varphi(D)\cap \varphi(D')$. Furthermore, $D\cup D'$ is a Condorcet domain if and only if $\varphi(D)\cap \varphi(D')$ is a Condorcet root poset.
\end{lemma}
\begin{proof}
    Suppose $\alpha \le \beta $ in $\varphi(D)\cap\varphi( D')$. This implies that $D(\alpha)\supseteq D(\beta)$ and $D'(\alpha)\supseteq D'(\beta)$. Consequently, $(D\cup D')(\alpha)=D(\alpha)\cup D'(\alpha)\supseteq D(\beta)\cup D'(\beta)=(D\cup D')(\beta)$. Thus, $\alpha \le \beta$ in $\varphi(D\cup D')$, which shows that $\varphi(D\cup D')\supseteq\varphi(D)\cap \varphi(D')$.  
    
    Conversely, by Lemma~\ref{lem: contravariant}, the inclusions $D, D'\subseteq D\cup D'$ imply that $\varphi(D\cup D') \subseteq \varphi(D)$ and $\varphi(D\cup D') \subseteq \varphi(D')$, so $\varphi(D\cup D') \subseteq \varphi(D)\cap \varphi(D')$. Therefore, $\varphi(D\cup D')=\varphi(D)\cap \varphi(D')$. The final statement follows directly from Proposition~\ref{prop: Condorcet domain eq map to Condorcet root poset}.
\end{proof}

\begin{proof}[Proof of Theorem~\ref{thm: connected eq peak-pit}]
    If $D$ is connected, then $D$ is peak-pit by Proposition~\ref{prop: connected implies peak-pit}. 
    
    Now assume that $D$ is peak-pit but not connected. By Proposition~\ref{prop: connected in poset}, there exist roots $\alpha_1, \alpha_2 \in \tilde{P}$ such that $\alpha_1 \neq_\Phi \alpha_2$ but $\alpha_1 = \alpha_2$. Let $A=\{ \gamma \in \Phi \mid \gamma =\alpha_1 \text{ in } P \}=\{\alpha_1, \alpha_2, \ldots, \alpha_l\}$. Then $|A|=l \ge 2$. By Definition~\ref{def: map phi}, the fact that $\alpha_1\in \tilde{P}$ implies there exist $v_1, v_2 \in D$ such that $\alpha_1 \in I(v_1)\setminus I(v_2)$ hence $A\subseteq I(v_1)\setminus I(v_2)$. Similarly to the proof of Proposition~\ref{prop: connected in poset}, we can choose $I, I'\in J_s(P)$ such that $I'\setminus I=A$. By Corollary~\ref{cor:bijection}, there exist $v, v' \in D$ such that $I(v)=I$ and $I(v')=I'$. 
    
    Because $\varphi(wD)$ is simply a relabeling of $\varphi(D)$, we may assume without loss of generality that $v'=w_0$. Consequently, $I(v')=\Phi^+$ and $A\subseteq \Phi^+$. Let $D'$ be the saturated single-crossing domain connecting $v$ and $v'$ guaranteed by Proposition~\ref{prop: saturated single-crossing domain exists}. Note that $D' \nsubseteq D$, since $v$ and $v'$ are not connected in $D$ but are connected in $D'$. To establish a contradiction with the maximality of $D$, it suffices to prove that $D\cup D'$ is a Condorcet domain. 
    
    By Lemma~\ref{lem: union to intersection}, letting $P'=\varphi(D')$, we need only show that $P'':=P\cap P'$ is a Condorcet root poset. Since the intersection of antipodal root posets is necessarily antipodal, by Lemma~\ref{lem:phi well-defined}, it remains only to verify the Condorcet condition for $P''$. 
    
    Take any root-triple $T=\{\alpha, \beta, \gamma\}$. Let $k=|T\cap \tilde{P'}|$. From the proof of Proposition~\ref{prop: saturated single-crossing domain exists} and the construction of $I$ and $I'$, we know that $\tilde{P'}$ is a poset on $A\sqcup -A$. Because $v'=w_0$ and $A\subseteq \Phi^+$, we can relabel the elements in $A$ to assume that the only relations in $\tilde{P'}$ form the chains $\alpha_1< \alpha_2<\cdots< \alpha_l$ and $-\alpha_1>-\alpha_2>\cdots>-\alpha_l$. As $A\subseteq \tilde{P},$ elements in $A$ are incomparable with elements in $-A$. Therefore, for any $\zeta_1, \zeta_2 \in A\sqcup-A$, we have $\zeta_1\le _{P''}\zeta_2$ if and only if $\zeta_1\le_{P'} \zeta_2$.  
    
    Furthermore, from the proof of Proposition~\ref{prop: saturated single-crossing domain exists} and $I(v), I(v')\in J_s(P)$, $P'_m= I(v)\cap I(v')$ is an order ideal in $P$, and $P'_M=\Phi\setminus(I(v)\cup I(v'))$ is an order filter in $P$. Moreover, we have $A\sqcup -A=I(v)\Delta I(v')=\tilde{P'}$ and $P'_m\sqcup P'_M\sqcup \tilde{P}'=\Phi$. So if $\zeta_1 \in P'_m, \zeta_2\in A\sqcup -A$, then $\zeta_1<_{P'}\zeta_2 $ and $\zeta_1\ngeq_P \zeta_2$. Similar analysis implies that for any $\{\zeta_1, \zeta_2\} \nsubseteq A\sqcup -A$, we have $\zeta_1\le_{P''}\zeta_2$ if and only if $\zeta_1\le_P \zeta_2$. 
    
    We now consider possible values of $k$:
    \begin{itemize}
        \item[(1)] When $k=0$ or $1$, the Condorcet condition on $T$ in $P''$ is equivalent to that in $P$. 
        \item[(2)] When $k=3$, the Condorcet condition on $T$ in $P''$ is equivalent to that in $P'$.
    \end{itemize}
    Since both $P$ and $P'$ are Condorcet root posets, the condition holds in $P''$ for these cases.  
    
    Now assume $k=2$. Without loss of generality, let $\alpha, \gamma \in \tilde{P'}$. Then $\beta \in P'_G$. Since $|T\cap \tilde{P'}|=|-T\cap \tilde{P'}|$ and the Condorcet conditions on $T$ and $-T$ are equivalent, we can replace $T$ with $-T$ if necessary to assume $\beta \in P'_m$.  As there are no equal roots in $\tilde{P'}$, the first type of Condorcet condition implies $-\beta >_{P'}\alpha >_{P'} -\gamma >_{P'} \beta$ and $-\beta >_{P'}\gamma >_{P'}-\alpha >_{P'}\beta$. As $\beta \in P'_m, -\beta \in P_M'$, we have $-\beta\nleq_P \alpha$ and $-\gamma\nleq_P\beta$. Because $P$ is peak-pit, the only possible first type of Condorcet conditions on $T$ in $P$ are $\alpha \ge_P -\gamma\ge_P \beta$ or $\gamma \ge_P -\alpha \ge_P \beta$. Both conditions are also satisfied in $P'$ and therefore the first type of Condorcet condition holds in $P''$.
    
    Thus, the Condorcet condition holds on $T$ in $P''$ in all cases, which completes the proof.
\end{proof}

Theorem~\ref{thm: connected eq peak-pit} was conjectured for type $A$ domains in \cite[Conjecture 1]{puppe2024maximal}, and was verified for $A_i$ with $i\le 6$ in \cite[Observation 4.15]{akello-egwelCondorcetDomainsMost2025}. Recently, \cite{liEquivalenceConnectedPeakPit2026} claimed to provide a proof for the type $A$ cases. 

The strategy for our proof is partly inspired by \cite{liEquivalenceConnectedPeakPit2026} and can be viewed as both a simplification and a generalization of their approach.

\subsection{Symmetric, of maximal width, and of tiling type} \label{subsec: Symmetry maximal width and tiling type}
The central focus of this subsection is the class of Condorcet domains of tiling type, which naturally combines the properties of being maximal, connected, and of maximal width. Although known examples demonstrate that this class does not always contain the maximum-size Condorcet domains, unlike the broader family of ample and connected domains discussed previously, it nonetheless possesses elegant properties and rich structural connections worthy of study. Because domains of tiling type are not a priori Condorcet domains by definition, we will temporarily broaden our focus to arbitrary nonempty subsets of $W$.

A nonempty subset $D \subseteq W$ is said to be \emph{symmetric} if $vw_0\in D$ for all $v\in D$, and \emph{of maximal width} if there exists at least one $v\in D$ such that $vw_0\in D$.

\begin{prop} 
    \label{prop: symmetry in poset}
    Let $D$ be a nonempty subset of $W$ and let $P=\varphi(D)$. If $D$ is symmetric, then $P$ coincides with an equivalent relation on $\Phi$, that is, any two non-equal elements $\alpha\neq \beta \in P$ are incomparable in $P$. Moreover, if $D$ is a closed Condorcet domain, the converse also holds.
\end{prop}
\begin{proof}
    Suppose $D$ is symmetric. Take any $\alpha\neq \beta \in P$, we will show that they are incomparable in $P$. We may assume without loss of generality that $\alpha \ngeq \beta$. By Definition~\ref{def: map phi}, there exists $v \in D$ such that $\alpha \in I(v)$ and $\beta \notin I(v)$. By symmetry, $vw_0 \in D$. Since $I(vw_0) = \Phi \setminus I(v)$, we have $\alpha \notin I(vw_0)$ and $\beta \in I(vw_0)$. This implies that $\beta \ngeq \alpha$ in $P$. Therefore, $\alpha$ and $\beta$ are incomparable.

    Conversely, suppose $D$ is a closed Condorcet domain and $P$ is an equivalent relation on $\Phi$. Let $I \in J_s(P)$. For any $\alpha \in -I$ and $\beta \notin- I$, we have $\alpha \ngeq \beta$, so $-I$ is an order ideal in $P$. As $I\sqcup -I=P$, $-I \in J_s(P)$. If $I = I(v)$, we note that $-I = \Phi \setminus I(v) = I(vw_0)$. Thus, by Corollary~\ref{cor:bijection} and Definition~\ref{def: map psi}, $v \in D$ implies $vw_0 \in D$, meaning $D$ is symmetric.
\end{proof}

The proof of the following proposition is analogous to that of Proposition~\ref{prop: symmetry in poset}, so we left it as an exercise.
\begin{prop}
    \label{prop: maximal width in poset}
   Let $D$ be a nonempty subset of $W$ and let $P=\varphi(D)$. If $D$ is of maximal width with $v, vw_0\in D$, then for the subposet $P^+=P|_{I(v)}$, we have $P=P^+\sqcup -P^+$, where the disjoint union means for any $\alpha \in P^+$ and $\beta \in -P^+$, $\alpha$ and $\beta$ are incomparable in $P$. Moreover, if $D$ is a closed Condorcet domain, the converse also holds.
\end{prop}

\begin{ex}
    \label{ex: maximum not maximal width}
    Let $P$ be the poset on the root system of type $A_7$ shown in Figure~\ref{fig:A_7 maximum}. Then $P$ is a Condorcet root poset, and $D=\psi(P)$ is the Condorcet domain of maximum size in $A_7$, which was found by \cite{leedham2024largest}. $D$ has size $224$. By Propositions~\ref{prop: ample in poset}, \ref{prop: connected in poset}, Theorem~\ref{thm: minimal cover is root-triple} and Proposition~\ref{prop: maximal width in poset}, we can see from $P$ that $D$ is ample, connected, maximal, but is not of maximal width. 
    
It was previously believed that, in type $A$, there exists a maximum-size Condorcet domain that is both ample and peak-pit (and hence connected). However, maximal width, another seemingly effective condition, was shown to fail for the maximum-size domain in type $A_7$ \cite{leedham2024largest}, which is the only counterexample we know in type $A$.

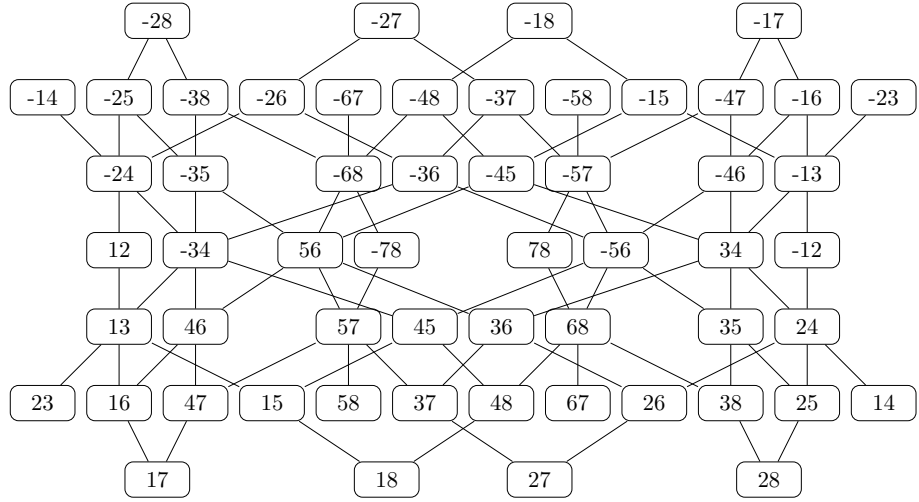
\begin{figure}[htbp]
    \centering
    \resizebox{0.8\textwidth}{!}{
        \begin{tikzpicture}[
            x=1.3cm, 
            y=1.3cm, 
            every node/.style={
                rectangle, 
                rounded corners=4pt, 
                draw, 
                fill=white, 
                inner sep=2pt, 
                minimum width=1.1cm, 
                minimum height=0.6cm
            }
        ]

        \node (n17) at (1.5, 0) {17};
        \node (n18) at (4.5, 0) {18};
        \node (n27) at (6.5, 0) {27};
        \node (n28) at (9.5, 0) {28};

        \node (n23) at (0, 1) {23};
        \node (n16) at (1, 1) {16};
        \node (n47) at (2, 1) {47};
        \node (n15) at (3, 1) {15};
        \node (n58) at (4, 1) {58};
        \node (n37) at (5, 1) {37};
        \node (n48) at (6, 1) {48};
        \node (n67) at (7, 1) {67};
        \node (n26) at (8, 1) {26};
        \node (n38) at (9, 1) {38};
        \node (n25) at (10, 1) {25};
        \node (n14) at (11, 1) {14};

        \node (n13) at (1, 2) {13};
        \node (n46) at (2, 2) {46};
        \node (n57) at (4, 2) {57};
        \node (n45) at (5, 2) {45};
        \node (n36) at (6, 2) {36};
        \node (n68) at (7, 2) {68};
        \node (n35) at (9, 2) {35};
        \node (n24) at (10, 2) {24};

        \node (n12) at (1, 3) {12};
        \node (n_34) at (2, 3) {-34};
        \node (n_78) at (4.5, 3) {-78};
        \node (n56) at (3.5, 3) {56};
        \node (n_56) at (7.5, 3) {-56};
        \node (n78) at (6.5, 3) {78};
        \node (n34) at (9, 3) {34};
        \node (n_12) at (10, 3) {-12};

        \node (n_24) at (1, 4) {-24};
        \node (n_35) at (2, 4) {-35};
        \node (n_68) at (4, 4) {-68};
        \node (n_36) at (5, 4) {-36};
        \node (n_45) at (6, 4) {-45};
        \node (n_57) at (7, 4) {-57};
        \node (n_46) at (9, 4) {-46};
        \node (n_13) at (10, 4) {-13};

        \node (n_14) at (0, 5) {-14};
        \node (n_25) at (1, 5) {-25};
        \node (n_38) at (2, 5) {-38};
        \node (n_26) at (3, 5) {-26};
        \node (n_67) at (4, 5) {-67};
        \node (n_48) at (5, 5) {-48};
        \node (n_37) at (6, 5) {-37};
        \node (n_58) at (7, 5) {-58};
        \node (n_15) at (8, 5) {-15};
        \node (n_47) at (9, 5) {-47};
        \node (n_16) at (10, 5) {-16};
        \node (n_23) at (11, 5) {-23};

        \node (n_28) at (1.5, 6) {-28};
        \node (n_27) at (4.5, 6) {-27};
        \node (n_18) at (6.5, 6) {-18};
        \node (n_17) at (9.5, 6) {-17};

        \draw (n17) -- (n16);
        \draw (n17) -- (n47);
        \draw (n18) -- (n15);
        \draw (n18) -- (n48);
        \draw (n27) -- (n37);
        \draw (n27) -- (n26);
        \draw (n28) -- (n38);
        \draw (n28) -- (n25);

        \draw (n23) -- (n13);
        \draw (n16) -- (n13);
        \draw (n16) -- (n46);
        \draw (n47) -- (n46);
        \draw (n47) -- (n57);
        \draw (n15) -- (n13);
        \draw (n15) -- (n45);
        \draw (n58) -- (n57);
        \draw (n37) -- (n57);
        \draw (n37) -- (n36);
        \draw (n48) -- (n45);
        \draw (n48) -- (n68);
        \draw (n67) -- (n68);
        \draw (n26) -- (n36);
        \draw (n26) -- (n24);
        \draw (n38) -- (n68);
        \draw (n38) -- (n35);
        \draw (n25) -- (n35);
        \draw (n25) -- (n24);
        \draw (n14) -- (n24);

        \draw (n13) -- (n12);
        \draw (n13) -- (n_34);
        \draw (n46) -- (n_34);
        \draw (n46) -- (n56);
        \draw (n57) -- (n_78);
        \draw (n57) -- (n56);
        \draw (n45) -- (n_34);
        \draw (n45) -- (n_56);
        \draw (n36) -- (n56);
        \draw (n36) -- (n34);
        \draw (n68) -- (n_56);
        \draw (n68) -- (n78);
        \draw (n35) -- (n34);
        \draw (n35) -- (n_56);
        \draw (n24) -- (n34);
        \draw (n24) -- (n_12);

        \draw (n12) -- (n_24);
        \draw (n_34) -- (n_24);
        \draw (n_34) -- (n_35);
        \draw (n_34) -- (n_36);
        \draw (n_78) -- (n_68);
        \draw (n56) -- (n_35);
        \draw (n56) -- (n_68);
        \draw (n56) -- (n_45);
        \draw (n78) -- (n_57);
        \draw (n_56) -- (n_36);
        \draw (n_56) -- (n_57);
        \draw (n_56) -- (n_46);
        \draw (n34) -- (n_13);
        \draw (n34) -- (n_46);
        \draw (n34) -- (n_45);
        \draw (n_12) -- (n_13);

        \draw (n_24) -- (n_14);
        \draw (n_24) -- (n_25);
        \draw (n_24) -- (n_26);
        \draw (n_35) -- (n_25);
        \draw (n_35) -- (n_38);
        \draw (n_36) -- (n_26);
        \draw (n_36) -- (n_37);
        \draw (n_68) -- (n_38);
        \draw (n_68) -- (n_67);
        \draw (n_68) -- (n_48);
        \draw (n_57) -- (n_37);
        \draw (n_57) -- (n_47);
        \draw (n_57) -- (n_58);
        \draw (n_45) -- (n_48);
        \draw (n_45) -- (n_15);
        \draw (n_46) -- (n_47);
        \draw (n_46) -- (n_16);
        \draw (n_13) -- (n_15);
        \draw (n_13) -- (n_16);
        \draw (n_13) -- (n_23);

        \draw (n_25) -- (n_28);
        \draw (n_38) -- (n_28);
        \draw (n_26) -- (n_27);
        \draw (n_37) -- (n_27);
        \draw (n_48) -- (n_18);
        \draw (n_15) -- (n_18);
        \draw (n_47) -- (n_17);
        \draw (n_16) -- (n_17);

        \end{tikzpicture}
    }
    \vspace{0.5em}
    \caption{The poset $P$ in Example~\ref{ex: maximum not maximal width}, which corresponding to the maximum-size but not maximal width Condorcet domain in type $A_7$}
    \label{fig:A_7 maximum}
\end{figure}
\end{ex}

Fix a reduced word $\mathbf{i}=(s_{i_1}, s_{i_2}, \ldots, s_{i_m})$ for $v\in W$. We say another reduced word $\mathbf{i}'=(s_{i'_1}, s_{i'_2},\ldots , s_{i'_m})$ for $v$ is in the \emph{commutation class} of $\mathbf{i}$, denoted by $C(\mathbf{i})$, if $\mathbf{i}'$ can be obtained from $\mathbf{i}$ by a sequence of \emph{commuting moves} (i.e., exchanging adjacent generators $s_i$ and $s_j$ whenever $s_is_j=s_js_i$).

Recall that any contiguous subsegment of a reduced word is again a reduced word; in particular, this holds for any prefix. A nonempty subset $D \subseteq W$ is said to be \emph{of tiling type} if there exists a reduced word $\mathbf{i}$ for $w_0$ such that $D$ consists of the elements whose reduced words are prefixes of reduced word in $C(\mathbf{i})$. That is,
$$D=\{s_{i'_1}s_{i'_2}\cdots s_{i'_l} \mid (s_{i'_1}, s_{i'_2}, \ldots, s_{i'_m})\in C(\mathbf{i}),\ 0\le l \le m\}.$$ 
In this case, we call $D$ the \emph{domain of tiling type represented by $\mathbf{i}$}.

For a reduced word $\mathbf{i}=(s_{i_1}, s_{i_2},\ldots, s_{i_m})$ for $v\in W$, the \emph{Heap poset} of $\mathbf{i}$, denoted by $H(\mathbf{i})$, is a poset whose underlying set corresponds to the positions in the word $\mathbf{i}$ (so that identical generators at different positions are treated as distinct elements). The partial order is generated by the relations $s_{i_s} < s_{i_t}$ whenever $s < t$ and the generators $s_{i_s}$ and $s_{i_t}$ do not commute.

\begin{ex}
    \label{ex: tiling type}
    In type $A_3$, consider the reduced word $\mathbf{i}=(s_1,s_3,s_2,s_1,s_3,s_2)$ for $w_0$. The Heap poset $H(\mathbf{i})$ corresponds to the poset shown in Figure~\ref{fig:Heap poset}. We can compute that the domain $D$ of tiling type represented by $\mathbf{i}$ is $\{1234, 2134, 1243, 2143, 2413, 4213, 2431, 4231, 4321\}$. For instance, the word $\mathbf{i}'=(s_1,s_3,s_2,s_3,s_1,s_2) \in C(\mathbf{i})$, and the reduced word $(s_1, s_3,s_2,s_3)$ for $2431 \in D$ is a prefix of $\mathbf{i}'$.
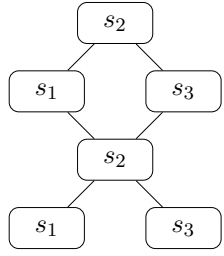
\begin{figure}[htbp]
    \centering
    \resizebox{0.2 \textwidth}{!}{
        \begin{tikzpicture}[
            x=1cm, 
            y=1cm, 
            every node/.style={
                rectangle, 
                rounded corners=4pt,
                draw, 
                fill=white, 
                inner sep=2pt, 
                minimum width=1.1cm, 
                minimum height=0.6cm
            }
        ]

        \node (n11) at (0, 0) {$s_1$};
        \node (n13) at (2, 0) {$s_3$};

        \node (n22) at (1, 1) {$s_2$};

        \node (n31) at (0, 2) {$s_1$};
        \node (n33) at (2, 2) {$s_3$};

        \node (n42) at (1, 3) {$s_2$};

        \draw (n11) -- (n22);
        \draw (n13) -- (n22);
        \draw (n22) -- (n31);
        \draw (n22) -- (n33);
        \draw (n31) -- (n42);
        \draw (n33) -- (n42);
        \end{tikzpicture}
    }
    \vspace{0.5em}
    \caption{The Heap poset $H(\mathbf{i)}$ in Example~\ref{ex: tiling type} with $\mathbf{i}=(s_1,s_3,s_2,s_1,s_3,s_2)$}
    \label{fig:Heap poset}
\end{figure}
\end{ex}

\begin{lemma}[{\cite[Proposition 2.2]{stembridgeFullyCommutativeElements1996}}]\label{lem: linear extension of heap}
Take any $v\in W$ and a reduced word $\mathbf{i}$ for $v$. The set of linear extensions of $H(\mathbf{i})$ is exactly $C(\mathbf{i})$.
\end{lemma}

\begin{prop} \label{prop: order ideal of heap}
If $D$ is the domain of tiling type represented by $\mathbf{i}$, then there exists a bijection $L:J(H(\mathbf{i})) \rightarrow D$, defined by mapping an order ideal $I$ of $H(\mathbf{i})$ to the group element for which any linear extension of $I$ is a reduced word.
\end{prop}
\begin{proof}
We first prove that $L$ is well-defined. Take any order ideal $I\in J(H(\mathbf{i}))$ and two linear extensions $l_1$ and $l_2$ of $I$. The sequence $l_2$ can be obtained from $l_1$ by repeatedly exchanging adjacent incomparable elements in $H(\mathbf{i})$. By the definition of the Heap poset, two elements $s_{i_s}, s_{i_t} \in H(\mathbf{i})$ being incomparable means they commute. Consequently, $l_1$ and $l_2$ is the reduced word for the same element in the Weyl group $W$. 

To show surjectivity, observe that by Lemma~\ref{lem: linear extension of heap} and the definition of $D$, the linear extensions of the order ideals of $H(\mathbf{i})$ correspond exactly to the prefixes of the linear extensions of whole $H(\mathbf{i})$, which precisely the reduced word for the elements of $D$.

For injectivity, suppose that $I_1, I_2\in J(H(\mathbf{i}))$ satisfy $L(I_1)=L(I_2)$. Let $I_0=I_1\cap I_2$, and let $l_0, l_1, l_2$ be linear extensions of $I_0$, $I_1\setminus I_0$, and $I_2\setminus I_0$, respectively. Let $v_i \in W$ be the element with reduced word $l_i$ for $i \in \{0, 1, 2\}$. Since $I_0$ is an order ideal in both $I_1$ and $I_2$, we can form valid linear extensions for $I_1$ and $I_2$ by concatenating $l_0$ with $l_1$, and $l_0$ with $l_2$, respectively. The assumption $L(I_1)=L(I_2)$ thus implies that $v_0v_1=v_0v_2$, and hence $v_1=v_2$. Furthermore, any element $s_{i_k}\in l_1$ and any element $s_{i_l}\in l_2$ must be incomparable in $H(\mathbf{i})$, if we had, for instance, $s_{i_k} < s_{i_l}$, then $s_{i_l} \in I_2$ would imply $s_{i_k} \in I_2$, contradicting $s_{i_k}\in I_1\setminus I_2$. Therefore, there exists a linear extension of $H(\mathbf{i})$ in which $s_{i_k}$ and $s_{i_l}$ appear adjacently. Since a reduced word cannot contain adjacent identical generators, the simple reflections corresponding to $s_{i_k}$ and $s_{i_l}$ must be distinct. Consequently, the element $v_1=v_2$ has two reduced words which use two disjoint subsets of the simple reflections $S$. This forces $v_1=v_2=e$. So $I_1\setminus I_2$ and $I_2\setminus I_1$ are empty and then $I_1=I_2$, completing the proof.
\end{proof}

\begin{ex}\label{ex: order ideal of heap}
Consider the domain $D$ from Example~\ref{ex: tiling type}. We see that $(s_1, s_3, s_2, s_1, s_3)$ and $(s_3, s_1, s_2, s_3, s_1)$ are both linear extensions of the order ideal $\{s_1, s_3, s_2, s_1, s_3\}$ in $H(\mathbf{i})$. Furthermore, they are reduced word for the same element $4231\in D$.
\end{ex}

Until we complete the proof of $(\Rightarrow)$ part of Theorem~\ref{thm: tiling type}, for a fixed reduced word $\mathbf{i}=(s_{i_1}, s_{i_2},\ldots, s_{i_m})$ for $w_0$, we use the following notations defined on $H(\mathbf{i})$. For any $1\le t\le m$, let $I_t$ be the principal order ideal in $H(\mathbf{i})$ generated by $s_{i_t}$, and let $I'_t := I_t\setminus\{s_{i_t}\}$. Set $v_t=L(I_t)$ and $v'_t=L(I'_t)$. Then $v'_t\lessdot v_t$ in the right weak order of $W$, and there exists a unique positive root $\alpha_t\in \Phi^+$ such that $\alpha_t \in I(v_t)\setminus I(v'_t)$. In fact, $\alpha_t$ is the positive root associated with the reflection $v'_ts_{i_t}{v'_t}^{-1}$. Denote the poset obtained by replacing each element $s_{i_t}$ in $H(\mathbf{i})$ with $\alpha_t$ by $P(\mathbf{i})$.  
\begin{lemma}
    \label{lem: alpha well defined}
    Take any reduced word $\mathbf{i}=(s_{i_1}, s_{i_2},\ldots, s_{i_m})$ for $w_0$ and $1\le t\le m$. Let $I\in J(H(\mathbf{i}))$ be any order ideal in which $s_{i_t}$ is a maximal element, and let $I'=I\setminus \{s_{i_t}\}\in J(H(\mathbf{i}))$. If $v=L(I)$ and $v'=L(I')$, then $\alpha_t$ is the unique positive root in $I(v)\setminus I(v')$. 
\end{lemma}
\begin{proof}
    Define $I'' = I\setminus I_t = I'\setminus I'_t$. Let $l''$ be a linear extension of $I''$ and let $v''$ be the element with reduced word $l''$. Because $s_{i_t}$ is incomparable with every element of $I''$ in $H(\mathbf{i})$, we have $v''s_{i_t} = s_{i_t}v''$ in $W$. The only positive root in $I(v)\setminus I(v')$ is the one associated with the reflection:
    $$v's_{i_t}{v'}^{-1} = v'_tv''s_{i_t}{v''}^{-1}{v'_t}^{-1} = v'_ts_{i_t}v''{v''}^{-1}{v'_t}^{-1} = v'_ts_{i_t}{v'_t}^{-1},$$
    which is exactly $\alpha_t$.
\end{proof}

\begin{lemma}
    \label{lem: inversion eq elements of order ideal}
    Take any reduced word $\mathbf{i}=(s_{i_1}, s_{i_2},\ldots, s_{i_m})$ for $w_0$. For any $I\in J(H(\mathbf{i}))$, let $v=L(I)$. Then $I(v)\cap \Phi^+=\{\alpha_t\mid s_{i_t}\in I\}$.
\end{lemma}
\begin{proof}
    This follows from the fact that $|I(v)\cap \Phi^+|=\ell(v)=|I|$, combined with Lemma~\ref{lem: alpha well defined} and induction on the size of $I$.
\end{proof}

\begin{prop}\label{prop: heap poset and tiling type}
If $D$ is the domain of tiling type represented by $\mathbf{i}=(s_{i_1}, s_{i_2},\ldots, s_{i_m})$ and $P=\varphi(D)$, then $P$ is of the form $P(\mathbf{i})\sqcup -P(\mathbf{i})$.
\end{prop}
\begin{proof}
Since $D$ is of tiling type, we have $e, w_0\in D$. By Proposition~\ref{prop: maximal width in poset}, it suffices to show that the restriction of $P|_{\Phi^+}$ coincides with $P(\mathbf{i})$. 
Taking $I=H(\mathbf{i})$ in Lemma~\ref{lem: inversion eq elements of order ideal}, we have $L(I)=w_0$ and $I(w_0)=\Phi^+$. Thus, the underlying set of the poset $P(\mathbf{i})$ is precisely $\Phi^+$. By Definition~\ref{def: map phi}, for any $\alpha, \beta \in \Phi^+$, we have $\alpha \le_P \beta$ if and only if $D(\alpha )=\{v\in D \mid \alpha \in I(v)\} \supseteq D(\beta )=\{v\in D\mid \beta \in I(v)\}$. By Proposition~\ref{prop: order ideal of heap} and Lemma~\ref{lem: inversion eq elements of order ideal}, this containment is equivalent to $J_\alpha:=\{I\in J(P(\mathbf{i}))\mid \alpha \in I\} \supseteq J_\beta :=\{ I\in J(P(\mathbf{i}))\mid \beta \in I\}$, which in turn means $\alpha \le_{P(\mathbf{i})} \beta$. This completes the proof.
\end{proof}

\begin{prop}
    \label{prop: tiling type is closed condorcet}
    If $D$ is a domain of tiling type represented by $\mathbf{i}$, then $D$ is a closed Condorcet domain.
\end{prop}
\begin{proof}
    Let $P=\varphi(D)$. By Corollary~\ref{cor:bijection}, it suffices to show that $P$ is a Condorcet root poset and that $\psi(P)=D$. From Proposition~\ref{prop: heap poset and tiling type}, we know that $P=P(\mathbf{i})\sqcup -P(\mathbf{i})$. Thus, for any order ideal $I\in J(P(\mathbf{i}))$, the set $I\sqcup -(P(\mathbf{i})\setminus I)$ is in $J_s(P)$. This establishes a bijection, which implies $|J_s(P)|=|J(P(\mathbf{i}))|$. By Lemma~\ref{lem: J_s(I)andJ_s(P)}, we have $J_s(P(\mathbf{i}))=J(P(\mathbf{i}))$. And then by Proposition~\ref{prop: order ideal of heap}, Lemma~\ref{lem: inversion eq elements of order ideal}, and the construction of $P(\mathbf{i})$, we obtain $\psi(P)=D$ and $|\psi(P)|=|J_s(P)|$. Therefore, $P$ is a Condorcet root poset by Theorem~\ref{thm: Condorcet root poset eq condition}, completing the proof.
\end{proof}

We are now justified in using the term \emph{Condorcet domain of tiling type}. Our definition generalizes the one given in \cite[Definition 1.1]{reiner2025majority}. In type $A$, Condorcet domains of tiling type admit several equivalent formulations, such as those discussed in \cite[Section 4]{danilov2012condorcet}, where the notion of tiling type was originally introduced. 

\begin{ex}
    \label{ex: tiling to Condorcet}
    Consider the domain $D$ from Example~\ref{ex: tiling type}. The corresponding poset $P=\varphi(D)$ can be computed and is shown in Figure~\ref{fig: Condorcet root poset of tiling}. We observe that $P$ is indeed a Condorcet root poset, which agrees with our construction in the proof of Proposition~\ref{prop: heap poset and tiling type}. Furthermore, invoking Theorem~\ref{thm: minimal cover is root-triple} and Propositions~\ref{prop: connected in poset} and \ref{prop: maximal width in poset}, we can verify directly from $P$ that $D$ is maximal, connected, and of maximal width. In fact, $D$ is the maximum-size Condorcet domain in type $A_3$. More generally, for $3\le i\le 6$, the maximum-size Condorcet domain in type $A_i$ is of tiling type.

\begin{figure}[htbp]
    \centering
    \resizebox{0.4 \textwidth}{!}{
        \begin{tikzpicture}[
            x=1cm, 
            y=1cm, 
            every node/.style={
                rectangle, 
                rounded corners=4pt,
                draw, 
                fill=white, 
                inner sep=2pt, 
                minimum width=1.1cm, 
                minimum height=0.6cm
            }
        ]

        \node (n11) at (0, 0) {12};
        \node (n13) at (2, 0) {34};

        \node (n22) at (1, 1) {14};

        \node (n31) at (0, 2) {24};
        \node (n33) at (2, 2) {13};

        \node (n42) at (1, 3) {23};
        
        \node (n-11) at (5, 3) {-12};
        \node (n-13) at (3, 3) {-34};

        \node (n-22) at (4, 2) {-14};

        \node (n-31) at (5, 1) {-24};
        \node (n-33) at (3, 1) {-13};

        \node (n-42) at (4, 0) {-23};
        
        \draw (n11) -- (n22);
        \draw (n13) -- (n22);
        \draw (n22) -- (n31);
        \draw (n22) -- (n33);
        \draw (n31) -- (n42);
        \draw (n33) -- (n42);

        \draw (n-11) -- (n-22);
        \draw (n-13) -- (n-22);
        \draw (n-22) -- (n-31);
        \draw (n-22) -- (n-33);
        \draw (n-31) -- (n-42);
        \draw (n-33) -- (n-42);
        \end{tikzpicture}
    }
    \vspace{0.5em}
    \caption{The poset $P$ in Example~\ref{ex: tiling to Condorcet}, which corresponding to the Condorcet domain $D$ of tiling type represented by $\mathbf{i}=(s_1,s_3,s_2,s_1,s_3,s_2)$}
    \label{fig: Condorcet root poset of tiling}
\end{figure}
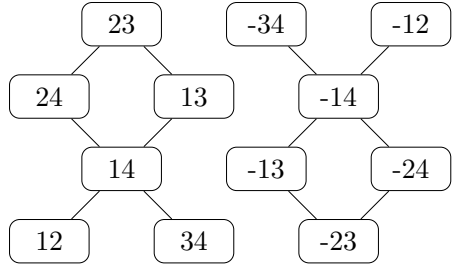
\end{ex}

\begin{proof}[Proof of Theorem~\ref{thm: tiling type}, $(\Rightarrow)$ direction.]
    Let $D$ be a closed Condorcet domain of tiling type represented by $\mathbf{i}=(s_{i_1}, s_{i_2},\ldots, s_{i_m})$. By Propositions~\ref{prop: connected in poset} and \ref{prop: maximal width in poset}, along with Proposition~\ref{prop: heap poset and tiling type} and the definition of the $P(\mathbf{i})$, $D$ is connected and of maximal width. Since $e, w_0\in D$ and $I(e)\sqcup I(w_0)=\Phi$, it immediately follows that $D$ is ample. 
    
    To prove that $D$ is maximal, Theorem~\ref{thm: minimal cover is root-triple} implies that for any covering relation $\alpha\lessdot \beta $ in $P$, we need only show that there is a root-triple containing $\alpha$ and $-\beta$. By Proposition~\ref{prop: heap poset and tiling type} and condition (1) of Definition~\ref{def: condorcet-poset}, we may assume without loss of generality that $\alpha, \beta \in P(\mathbf{i})$, and thus $\alpha, \beta \in \Phi^+$. 
    
    Let $\alpha=\alpha_s$ and $\beta=\alpha_t$ for $1\le s< t\le m$. Define $I'' = I_t\setminus I_s = I'_t\setminus I'_s$. Let $l''$ be a linear extension of $I''$, and let $v''$ be the group element with reduced word $l''$. Because $s_{i_s}\lessdot s_{i_t}$ in $H(\mathbf{i})$, $s_{i_s}$ is incomparable with every element in $I''$, meaning that their corresponding simple reflections commute, thus, $v''s_{i_s}=s_{i_s}v''$ in $W$. The positive roots $\alpha_s$ and $\alpha_t$ correspond to the reflections $r_s:=v_s's_{i_s}{v'_s}^{-1}$ and $r_t:=v_t's_{i_t}{v'_t}^{-1}=v'_ss_{i_s}v''s_{i_t}{v''}^{-1}s_{i_s}{v'_s}^{-1}$. Thus, we compute the commutator:
    \begin{align*}
        r_sr_tr_s^{-1}r_t^{-1} &= (v'_ss_{i_s}{v'_s}^{-1}) \cdot (v'_ss_{i_s}v''s_{i_t}{v''}^{-1}s_{i_s}{v'_s}^{-1}) \cdot (v'_ss_{i_s}{v'_s}^{-1}) \cdot (v'_ss_{i_s}v''s_{i_t}{v''}^{-1}s_{i_s}{v'_s}^{-1})\\
        &= v'_sv''s_{i_t}{v''}^{-1}s_{i_s}v''s_{i_t}{v''}^{-1}s_{i_s}{v'_s}^{-1}\\
        &= v'_sv''s_{i_t}s_{i_s}s_{i_t}s_{i_s}{v''}^{-1}{v'_s}^{-1}.
    \end{align*}
    Consequently, if $r_sr_tr_s^{-1}r_t^{-1}=e$, it would imply that $s_{i_t}s_{i_s}s_{i_t}s_{i_s}=e$, meaning $s_{i_s}$ and $s_{i_t}$ commute. However, because $s_{i_s}\lessdot s_{i_t}$, the generators $s_{i_s}$ and $s_{i_t}$ do not commute. Therefore, $r_sr_tr_s^{-1}r_t^{-1}\neq e$, which means the roots $\alpha_s$ and $\alpha_t$ are not orthogonal. 
    
    Consider the dihedral root cycle $(\gamma_1, \gamma_2,\ldots, \gamma_{2l})$ in the plane spanned by $\alpha_s$ and $\alpha_t$. Without loss of generality, we may assume $\alpha_s=\gamma_1$ and $\alpha_t=\gamma_k$ for some $1<k\le l$. Since $\alpha_s$ and $\alpha_t$ are not orthogonal, we must have $l\ge 3$. 
    
    If $k<l$, then the set $\{\gamma_1, \gamma_l, \gamma_{k+l}\} = \{\alpha_s, \gamma_l, -\alpha_t\}$ forms a root-triple containing $\alpha_s$ and $-\alpha_t$, and the claim holds. 
    
    If $k=l$, then for $1\le i \le l$, $\gamma_i\in \Phi^+$. By Theorem~\ref{thm: cycle condition}, the construction of $P$, and the fact that $\alpha_s\lessdot \alpha_t$ in $\Phi^+$, we must have the chain $\alpha_s=\gamma_1<\gamma_2< \cdots <\gamma_l=\alpha_t$ in $P$. Because $\alpha_s \lessdot \alpha_t$ is a covering relation, this forces $l \le 2$, which contradicts the fact that $l\ge 3$ and we are done.
    \end{proof}
   In the rest of this subsection, we prove the $(\Leftarrow)$ direction. We fix a closed Condorcet domain $D$ which is of maximal width, connected, and maximal, with $e, w_0\in D$. Then Proposition~\ref{prop: maximal width in poset}, combined with the fact that $e, w_0\in D$, implies that $P=P^+\sqcup -P^+$, where $P^+=P|_{\Phi^+}$. Moreover, $I(e)\sqcup I(w_0)=\Phi$ implies $D$ is ample and $P=\tilde{P}$. And then Proposition~\ref{prop: connected in poset} implies that there are no equal roots in $P$. 

Let $\Phi^+=\{\alpha_1,\alpha_2, \ldots, \alpha_m\}$, where $m=|\Phi^+|$. Until we complete the proof of $(\Leftarrow)$ part of Theorem~\ref{thm: tiling type}, we use the following notations defined on $P^+$.  For each $1\le t \le m$, let $I_t$ be the principal order ideal generated by $\alpha_t$ in $P^+$, and let $I'_t=I_t\setminus \{\alpha_t\}$. Since $I\sqcup -(P^+\setminus I)\in J_s(P)$ for any $I\in J(P^+)$, we have a identification $J_s(P^+)=J(P^+)$ by Lemma~\ref{lem: J_s(I)andJ_s(P)}. Thus, $I_t, I'_t\in J(P^+) = J_s(I(w_0))$, and we can define $v_t=\rho^{-1}(I_t)$ and $v'_t=\rho^{-1}(I'_t)$ in $D$. This implies $I(v_t)\cap \Phi^+= (I(v'_t)\cap \Phi^+)\sqcup \{\alpha_t\}$ and $v'_t\lessdot v_t$ in the right weak order of $W$. Thus, there exists a unique simple reflection $s_{i_t}\in S$ such that $v_t=v'_ts_{i_t}$. Let $H$ be the poset obtained by replacing each $\alpha_t$ with $s_{i_t}$ in $P^+$ for $1\le t\le m$, and let $R: P^+\rightarrow H$, defined by $\alpha_t\mapsto s_{i_t}$, be the corresponding order isomorphism. We will prove that $H$ is the Heap poset of some reduced word $\mathbf{i}$, and that $D$ is the Condorcet domain of tiling type represented by $\mathbf{i}$. We first establish two lemmas.

\begin{lemma}
    \label{lem: linear extension is reduced word}
    For any $I\in J(P^+)$, let $v=\rho^{-1}(I)$. Then any linear extension of $R(I)$ is a reduced word for $v$.
\end{lemma}
\begin{proof}
    We proceed by induction on $|I|$. If $|I|=0$, then $v=e$, and its reduced word is empty, thus the base case holds. 
    
    Assume the induction hypothesis holds for all order ideals of size up to some $k$ with $0\le k < m$, and let $I\in J(P^+)$ with $|I|=k+1$. If $I$ has a unique maximal element, then $I=I_t$ for some $1\le t\le m$. Any linear extension of $R(I)$ is a linear extension of $R(I'_t)$ with $s_{i_t}$ appended to the end. The result then follows from the definition of $s_{i_t}$ and the induction hypothesis. 
    
    If $I$ has multiple maximal elements, let $l_1$ and $l_2$ be any two linear extensions of $R(I)$. If $l_1$ and $l_2$ end with the same element $s_{i_t}$, we are done by applying the induction hypothesis to $I\setminus \{\alpha_t\}$. Suppose instead that $l_1$ and $l_2$ end with $s_{i_t}$ and $s_{i_s}$, respectively, where $s\neq t$. Then $\alpha_t$ and $\alpha_s$ are both maximal elements in $I$. Since the induction hypothesis guarantees that all linear extensions of $R(I\setminus \{\alpha_s\})$ and $ R(I\setminus \{\alpha_t\})$ are reduced word for $\rho^{-1}(I\setminus \{\alpha_s\})$ and $\rho^{-1}(I\setminus\{\alpha_t\})$, respectively. We may assume without loss of generality that $l_1$ and $l_2$ are of the form $(l', s_{i_s}, s_{i_t})$ and $(l', s_{i_t}, s_{i_s})$, respectively, where $l'$ is a linear extension of $R(I\setminus \{\alpha_s, \alpha_t\})$. To show that $l_1$ and $l_2$ are reduced word for the same element, it suffices to prove that $s_{i_s}s_{i_t}=s_{i_t}s_{i_s}$. 
    
    Since $\alpha_t$ and $\alpha_s$ are maximal in $I$, they are incomparable in $P^+$. Let $(\gamma_1, \gamma_2,\ldots, \gamma_{2l})$ be the dihedral root cycle in the plane spanned by $\alpha_t$ and $\alpha_s$. Assume without loss of generality that $\gamma_k\in \Phi^+$ for $1\le k\le l$. If $l\ge 3$, then since there are no equal roots in $P$, Theorem~\ref{thm: cycle condition} implies that either $\gamma_1<\gamma_2<\cdots<\gamma_l$ or $\gamma_1>\gamma_2>\cdots>\gamma_l$. This forces $\alpha_t$ and $\alpha_s$ to be comparable, which is a contradiction. 
    
    Thus $l=2$, which means $\alpha_t$ and $\alpha_s$ are orthogonal. Let $u$ be the group element with reduced word $l'$. Then induction hypothesis on $I\setminus \{\alpha_s\}$, $ I\setminus \{\alpha_t\}$ and $I\setminus \{\alpha_s, \alpha_t\}$ imply that $\alpha_t$ and $\alpha_s$ correspond to the reflections $us_{i_t}u^{-1}$ and $us_{i_s}u^{-1}$, respectively. The orthogonality $\alpha_t \perp \alpha_s$ implies that their corresponding reflections commute:
    $$ (us_{i_t}u^{-1}) \cdot (us_{i_s}u^{-1}) \cdot (us_{i_t}u^{-1}) \cdot (us_{i_s}u^{-1}) = u (s_{i_t}s_{i_s}s_{i_t}s_{i_s}) u^{-1} = e.$$
    Therefore, $s_{i_t}s_{i_s}s_{i_t}s_{i_s}=e$, which implies $s_{i_s}s_{i_t}=s_{i_t}s_{i_s}$, completing the inductive step.
\end{proof}

\begin{lemma}
    \label{lem: H is heap}
    Let $\mathbf{i}$ be any linear extension of $H$. Then $\mathbf{i}$ is a reduced word for $w_0$ and $H=H(\mathbf{i})$.
\end{lemma}
\begin{proof}
    That $\mathbf{i}$ is a reduced word for $w_0$ follows directly from the facts that $H=R(P^+)$, $\rho^{-1}(P^+)=w_0$, and Lemma~\ref{lem: linear extension is reduced word}. By relabeling the elements of $\Phi^+$ if necessary, we may assume without loss of generality that $\mathbf{i}=(s_{i_1}, s_{i_2},\ldots, s_{i_m})$. 
    
    To prove that $H=H(\mathbf{i})$, we must show that the partial order on $H$ is generated exactly by the relations $s_{i_s}<s_{i_t}$ for all $s < t$ such that $s_{i_s}s_{i_t}\neq s_{i_t}s_{i_s}$. 
    
    First, suppose $s<t$ and $s_{i_s}s_{i_t}\neq s_{i_t}s_{i_s}$, we must prove that $s_{i_s}< s_{i_t}$ in $H$. From the proof of Lemma~\ref{lem: linear extension is reduced word}, consider the order ideal generated by $\alpha_s, \alpha_t$ in $P^+$, we have that $\alpha_s$ and $\alpha_t$ are incomparable in $P^+$ implies $s_{i_s}s_{i_t}= s_{i_t}s_{i_s}$. Thus $s_{i_s}s_{i_t}\neq s_{i_t}s_{i_s}$ implies that $\alpha_s$ and $\alpha_t$ are comparable in $P^+$, and thus $s_{i_s}$ and $s_{i_t}$ are comparable in $H$. Since $s_{i_s}$ appears before $s_{i_t}$ in the linear extension $\mathbf{i}$, we cannot have $s_{i_s} > s_{i_t}$. Hence, $s_{i_s}< s_{i_t}$. 
    
    Conversely, suppose $s_{i_s}\lessdot s_{i_t}$ is a covering relation in $H$, we must prove that $s<t$ and $s_{i_s}s_{i_t}\neq s_{i_t}s_{i_s}$. The fact that $s<t$ is immediate since $\mathbf{i}$ is a linear extension of $H$ and $s_{i_s}<s_{i_t}$. By Theorem~\ref{thm: minimal cover is root-triple}, the corresponding covering relation $\alpha_s\lessdot \alpha_t$ in $P^+$ implies that there is a root-triple containing $\alpha_s$ and $-\alpha_t$. Let $(\gamma_1, \gamma_2,\ldots, \gamma_{2l})$ be the dihedral root cycle in the plane spanned by $\alpha_s$ and $\alpha_t$, and assume $\gamma_k\in \Phi^+$ for $1\le k\le l$. Since the roots in a root-triple are coplanar, we must have $l\ge 3$. Then Theorem~\ref{thm: cycle condition}, together with the facts that there are no equal roots in $P^+$ and $\gamma_k\in P^+$ for $1\le k \le l$, imply that either $\gamma_1<\gamma_2<\cdots<\gamma_l$ or $\gamma_1>\gamma_2>\cdots>\gamma_l$. Thus, we can assume $\alpha_s=\gamma_j$ and $\alpha_t=\gamma_{j+1}$ for some $1\le j\le l-1$. Because $l\ge 3$, $\alpha_s$ and $\alpha_t$ are not orthogonal. 
    
    Since $\alpha_s \lessdot \alpha_t$, $\alpha_s$ is a maximal element in $I'_t$ and $I'':=I'_t\setminus \{\alpha_s\}$ is an order ideal in $P^+$. Let $u$ be the group element whose reduced word is a linear extension of $R(I'')$. By Lemma~\ref{lem: linear extension is reduced word} on $I''$, $I_t'$ and $I_t$, the positive roots $\alpha_s$ and $\alpha_t$ correspond to the reflections $r_s:=us_{i_s}u^{-1}$ and $r_t:=u s_{i_s}s_{i_t}s_{i_s}u^{-1}$. The fact that $\alpha_s$ and $\alpha_t$ are not orthogonal implies that their reflections do not commute:
    \begin{align*}
        r_sr_tr_s^{-1}r_t^{-1} &= (u s_{i_s}u^{-1}) \cdot (u s_{i_s}s_{i_t}s_{i_s}u^{-1}) \cdot (u s_{i_s}u^{-1}) \cdot (u s_{i_s}s_{i_t}s_{i_s}u^{-1})\\
        &= u (s_{i_t}s_{i_s}s_{i_t}s_{i_s}) u^{-1}\neq e.
    \end{align*}
    Thus $s_{i_t}s_{i_s}s_{i_t}s_{i_s}\neq e$, which means $s_{i_s}s_{i_t}\neq s_{i_t}s_{i_s}$, completing the proof.
\end{proof}
\begin{proof}[Proof of Theorem~\ref{thm: tiling type}, $(\Leftarrow)$ direction.]
        Now, for the $D$ we have fixed,  let $\mathbf{i}$ be the reduced word for $w_0$ in Lemma~\ref{lem: H is heap} and $D'$ be the Condorcet domain of tiling type represented by $\mathbf{i}$. By Lemma~\ref{lem: H is heap}, $H=H(\mathbf{i})$. Let $L$ be the bijection from $J(H(\mathbf{i}))=J(H)$ to $D'$ defined in Proposition~\ref{prop: order ideal of heap}. We have $\rho: v\mapsto I(v)\cap P^+$ is a bijection from $D$ to $J(P^+)$, and then $R\circ \rho$ is a bijection from $D$ to $J(H)$. So the composition $T:=L\circ R\circ \rho$ is a bijection from $D$ to $D'$. By Lemma~\ref{lem: linear extension is reduced word}, any linear extension of $R\circ \rho(v)$ is a reduced word for $v$, which is also a reduced word for $T(v)$ by the definition of $L$. Therefore, $v = T(v)$, which implies $D=D'$, completing the proof.
\end{proof}

\subsection{Restriction to root subsystems} \label{subsec: Restriction to root subsystems}
Let $\Phi'\subseteq \Phi$ be a root subsystem, as any root-triple $\{\alpha, \beta, \gamma\}$ in $\Phi'$ is also a root-triple in $\Phi$, for any biconvex set $I\subseteq \Phi$, the intersection $I\cap \Phi'$ is biconvex in $\Phi'$. We can therefore define the \emph{restriction} to $\Phi'$ of any subset $D \subseteq W$ as
$$D':=\{ w\in W'=W(\Phi')\mid \text{there exists } v\in D \text{ such that } I(w)=I(v)\cap \Phi'\}.$$ 
Since any root-triple $\{\alpha, \beta, \gamma\}$ in $\Phi'$ is a root-triple in $\Phi$, restriction preserves the property of being a Condorcet domain by Definition~\ref{def: condorcet}. Furthermore, restriction also preserves being closed.

\begin{prop}
    \label{prop: restriction preserve closed}
     If $D'$ is the restriction of a closed Condorcet domain $D$ to $\Phi'$, then $D'$ is also closed.
\end{prop}
\begin{proof}
    Let $P=\varphi(D)$ and $P'$ be the restriction of the poset $P$ to the underlying set $\Phi'$. Then $P'$ is clearly a Condorcet root poset by Definition~\ref{def: condorcet-poset}. By Corollary~\ref{cor:bijection}, it suffices to prove that $\psi(P')=D'$. 
    
    For any $w\in D'$, there exists $v\in D$ such that $I(w)=I(v)\cap \Phi'$. Since $I(v)$ is a symmetric order ideal in $P$, its intersection with $\Phi'$ is a symmetric order ideal in $P'$. Thus, $w \in \psi(P')$, which establishes $D' \subseteq \psi(P')$.
    
    To prove the reverse inclusion, $\psi(P')\subseteq D'$, take any $I'\in J_s(P')$. We wish to find $I\in J_s(P)$ such that $I\cap \Phi'=I'$. Let $I_0$ be the order ideal generated by $I'$ in $P$. We claim that $-I_0\cap I_0=\emptyset$. To see this, note that $-I_0$ is exactly the order filter generated by $-I'$ in $P$. If their intersection were non-empty, there would exist $\alpha \in I'$ and $\beta\in -I'$ such that $\beta \le_P \alpha$. Because both $\alpha, \beta \in \Phi'$, this would mean $\beta \le_{P'} \alpha$, contradicting the fact that $I'$ is an symmetric order ideal in $P'$. 
    
    Therefore, $I_0\in J_p(P)$. By Proposition~\ref{prop: J_p(P) in J_s(P)}, there exists $I\in J_s(P)$ such that $I\supseteq I_0\supseteq I'$. Because $-I'\cap I\subseteq -I_0\cap I \subseteq -I\cap I=\emptyset$, and $I' \sqcup -I' = \Phi'$, it necessarily follows that $I\cap \Phi'=I'$. This completes the proof.
\end{proof}

Note that a domain being ample can equivalently be defined as its restriction to any one-dimensional root subsystem of $\Phi$ being maximal. Consequently, restriction does not preserve maximality in general, as illustrated by Example~\ref{ex: maximal doesn't imply ample}. However, by Propositions~\ref{prop: ample in poset}, \ref{prop: connected in poset}, \ref{prop: symmetry in poset}, and \ref{prop: maximal width in poset}, along with the definitions of saturated single-crossing and peak-pit domains, the properties of being ample, connected, saturated single-crossing, peak-pit, symmetric, and of maximal width are all preserved under restriction.

\section{Voting in Condorcet domains}\label{sec: voting}
We devote a whole section to a short discussion on voting in Condorcet domains, which is the origin of the classical theory on Condorcet domains, and is also our initial motivation. Define the \emph{support} of a voting profile $f$ on $D$ to be $\operatorname{supp}(f)=\{v\in D\mid f(v)> 0\}$. In the definition of $\nu(D)$, we required the voting profile $f$ to have no ties. This condition is indeed necessary, as we can see in Example~\ref{ex: Condorcet and voting} (2). However, there are certain cases where $f$ has ties, yet we still obtain $R(f)\subseteq \nu(D)$. We say that a voting profile $f$ has only \emph{simple ties} if $R(f)\subseteq \nu(\operatorname{supp}(f))$. 

\begin{proof}[Proof of Theorem~\ref{thm: simple ties}]
    Let $v\in R(f)$, by Definition~\ref{def: closed}, we need only show $v\in D$. Let $A=\{\alpha\in \Phi \mid f(\alpha)=|f|/2\} \subseteq I(f)$ to be the set of tied roots. Let $P=\varphi(D)$. Recall that for any $\alpha \in P_G$, either $\alpha \in I(u)$ for all $u \in D$ or $\alpha \notin I(u)$ for all $u \in D$. Thus, $f(\alpha)=|f|$ or $f(\alpha)=0$. Since $|f|\ge |D|>0$, it follows that $A\cap P_G=\emptyset$. Furthermore, because $D$ is connected, Proposition~\ref{prop: connected in poset} ensures there are no equal roots in $\tilde{P}$, and therefore none in $A$. 
    
    Now, consider any strict relation $\alpha > \beta $ in $P=\varphi(D)$. By Definition~\ref{def: map phi}, this means $D(\alpha) \subsetneq D(\beta)$, so there exists $u\in D$ such that $\beta \in I(u)$ but $\alpha \notin I(u)$. Since $u \in \operatorname{supp}(f)$, meaning $f(u)>0$, this strict inclusion implies $f(\alpha) < f(\beta)$. Consequently, no two elements in $A$ can be comparable in $P$, meaning $A$ is an antichain. 
    Moreover, the sets $I(f)=\{\alpha\in \Phi\mid f(\alpha)\ge|f|/2\}$ and $I(f)\setminus A=\{\alpha\in \Phi\mid f(\alpha)>|f|/2\}$ are both order ideals in $P$. This implies that $A$ consists of maximal elements in the order ideal $I(f)$. 
    
    Because $v \in R(f)$ is an outcome of the voting profile, its inversion set $I(v)$ must agree with the strict majority and break ties in $A$. Thus, $I(v)=I(f)\setminus (A\setminus I(v))$. Since $A$ is an antichain of maximal elements in $I(f)$, removing any subset of $A$ leaves a valid order ideal. Therefore, $I(v)\in J_s(P)$. By Corollary~\ref{cor:bijection}, this implies $v \in \psi(P) = D$, which completes the proof.
\end{proof}

\section*{Acknowledgments}
We thank Prof. Vic Reiner for helpful discussions. Y.G. is partially supported by NSFC Grant no. 12471309.

\bibliographystyle{plain}
\bibliography{ref}

@article {Puppe-Slinko-median-graph,
    AUTHOR = {Puppe, Clemens and Slinko, Arkadii},
     TITLE = {Condorcet domains, median graphs and the single-crossing
              property},
   JOURNAL = {Econom. Theory},
  FJOURNAL = {Economic Theory},
    VOLUME = {67},
      YEAR = {2019},
    NUMBER = {1},
     PAGES = {285--318},
      ISSN = {0938-2259,1432-0479},
   MRCLASS = {91B12 (05C90 91B08)},
  MRNUMBER = {3905426},
MRREVIEWER = {Daniel\ Eckert},
       DOI = {10.1007/s00199-017-1084-6},
       URL = {https://doi.org/10.1007/s00199-017-1084-6},
}

@incollection{monjardet2009acyclic,
    AUTHOR = {Monjardet, Bernard},
     TITLE = {Acyclic domains of linear orders: {A} survey},
 BOOKTITLE = {The mathematics of preference, choice and order},
    SERIES = {Stud. Choice Welf.},
     PAGES = {139--160},
 PUBLISHER = {Springer, Berlin},
      YEAR = {2009},
      ISBN = {978-3-540-79127-0},
   MRCLASS = {91B14},
  MRNUMBER = {2648300},
       DOI = {10.1007/978-3-540-79128-7\_8},
       URL = {https://doi.org/10.1007/978-3-540-79128-7_8},
}

@article{danilov2012condorcet,
    AUTHOR = {Danilov, Vladimir I. and Karzanov, Alexander V. and Koshevoy,
              Gleb},
     TITLE = {Condorcet domains of tiling type},
   JOURNAL = {Discrete Appl. Math.},
  FJOURNAL = {Discrete Applied Mathematics. The Journal of Combinatorial
              Algorithms, Informatics and Computational Sciences},
    VOLUME = {160},
      YEAR = {2012},
    NUMBER = {7-8},
     PAGES = {933--940},
      ISSN = {0166-218X,1872-6771},
   MRCLASS = {91B14 (05B45 91B08)},
  MRNUMBER = {2901112},
       DOI = {10.1016/j.dam.2011.08.001},
       URL = {https://doi.org/10.1016/j.dam.2011.08.001},
}

@article{puppe2024maximal,
    AUTHOR = {Puppe, Clemens and Slinko, Arkadii},
     TITLE = {Maximal {C}ondorcet domains a further progress report},
   JOURNAL = {Games Econom. Behav.},
  FJOURNAL = {Games and Economic Behavior},
    VOLUME = {145},
      YEAR = {2024},
     PAGES = {426--450},
      ISSN = {0899-8256,1090-2473},
   MRCLASS = {91B14 (91B08)},
  MRNUMBER = {4732147},
MRREVIEWER = {Stefano\ Vannucci},
       DOI = {10.1016/j.geb.2024.04.001},
       URL = {https://doi.org/10.1016/j.geb.2024.04.001},
}

@article{reiner2025majority,
  title={Majority relations for Condorcet domains of tiling type},
  author={Reiner, Victor and Tenner, Bridget Eileen},
  journal={arXiv preprint arXiv:2509.19614},
  year={2025}
}

@article{leedham2024largest,
    AUTHOR = {Leedham-Green, Charles and Markstr\"om, Klas and Riis, S\o
              ren},
     TITLE = {The largest {C}ondorcet domain on 8 alternatives},
   JOURNAL = {Soc. Choice Welf.},
  FJOURNAL = {Social Choice and Welfare},
    VOLUME = {62},
      YEAR = {2024},
    NUMBER = {1},
     PAGES = {109--116},
      ISSN = {0176-1714,1432-217X},
   MRCLASS = {91B08 (91B14)},
  MRNUMBER = {4695881},
       DOI = {10.1007/s00355-023-01481-3},
       URL = {https://doi.org/10.1007/s00355-023-01481-3},
}

@article {Dyer-weak,
    AUTHOR = {Dyer, Matthew},
     TITLE = {On the weak order of {C}oxeter groups},
   JOURNAL = {Canad. J. Math.},
  FJOURNAL = {Canadian Journal of Mathematics. Journal Canadien de
              Math\'ematiques},
    VOLUME = {71},
      YEAR = {2019},
    NUMBER = {2},
     PAGES = {299--336},
      ISSN = {0008-414X,1496-4279},
   MRCLASS = {20F55 (06B23 17B22)},
  MRNUMBER = {3943754},
MRREVIEWER = {Jean\ Fromentin},
       DOI = {10.4153/cjm-2017-059-0},
       URL = {https://doi.org/10.4153/cjm-2017-059-0},
}

@article{akello-egwelCondorcetDomainsMost2025,
  title = {Condorcet Domains on at Most Seven Alternatives},
  author = {{Akello-Egwel}, Dolica and {Leedham-Green}, Charles and Litterick, Alastair and Markstr{\"o}m, Klas and Riis, S{\o}ren},
  year = 2025,
  month = jan,
  journal = {Mathematical Social Sciences},
  volume = {133},
  pages = {23--33},
  issn = {01654896},
  doi = {10.1016/j.mathsocsci.2024.12.002},
  urldate = {2026-02-20},
  langid = {english}
}

@article{stembridgeFullyCommutativeElements1996,
  title = {On the {{Fully Commutative Elements}} of {{Coxeter Groups}}},
  author = {Stembridge, John R.},
  year = 1996,
  month = oct,
  journal = {Journal of Algebraic Combinatorics},
  volume = {5},
  number = {4},
  pages = {353--385},
  issn = {1572-9192},
  doi = {10.1023/A:1022452717148},
  urldate = {2026-03-25},
  langid = {english}
}

@book{stanleyEnumerativeCombinatoricsVolume2012,
  title = {Enumerative Combinatorics. {{Volume}} 1},
  author = {Stanley, Richard P.},
  year = 2012,
  series = {Cambridge {{Studies}} in {{Advanced Mathematics}}},
  edition = {Second},
  volume = {49},
  publisher = {Cambridge University Press, Cambridge},
  isbn = {978-1-107-60262-5},
  mrnumber = {2868112}
}

@inproceedings{liEquivalenceConnectedPeakPit2026,
  title = {Equivalence of~{{Connected}} and~{{Peak-Pit Maximal Condorcet Domains}}},
  booktitle = {Computing and {{Combinatorics}}},
  author = {Li, Guanhao},
  editor = {Fomin, Fedor V. and Xiao, Mingyu},
  year = 2026,
  pages = {307--319},
  publisher = {Springer Nature},
  address = {Singapore},
  doi = {10.1007/978-981-95-0215-8_23},
  isbn = {9789819502158},
  langid = {english}
}

@incollection {bandelt2008metric,
    AUTHOR = {Bandelt, Hans-J\"urgen and Chepoi, Victor},
     TITLE = {Metric graph theory and geometry: a survey},
 BOOKTITLE = {Surveys on discrete and computational geometry},
    SERIES = {Contemp. Math.},
    VOLUME = {453},
     PAGES = {49--86},
 PUBLISHER = {Amer. Math. Soc., Providence, RI},
      YEAR = {2008},
      ISBN = {978-0-8218-4239-3},
   MRCLASS = {05C12 (05C62 51K05 52-02 54E35)},
  MRNUMBER = {2405677},
MRREVIEWER = {Mikhail\ Ostrovskii},
       DOI = {10.1090/conm/453/08795},
       URL = {https://doi.org/10.1090/conm/453/08795},
}

@article {klavzar1999median,
    AUTHOR = {Klav\v zar, Sandi and Mulder, Henry Martyn},
     TITLE = {Median graphs: characterizations, location theory and related
              structures},
   JOURNAL = {J. Combin. Math. Combin. Comput.},
  FJOURNAL = {Journal of Combinatorial Mathematics and Combinatorial
              Computing},
    VOLUME = {30},
      YEAR = {1999},
     PAGES = {103--127},
      ISSN = {0835-3026,2817-576X},
   MRCLASS = {05C12},
  MRNUMBER = {1705337},
MRREVIEWER = {Pranava\ K.\ Jha},
}

@book {mulder1980interval,
    AUTHOR = {Mulder, H. M.},
     TITLE = {The interval function of a graph},
    SERIES = {Mathematical Centre Tracts},
    VOLUME = {132},
 PUBLISHER = {Mathematisch Centrum, Amsterdam},
      YEAR = {1980},
     PAGES = {iii+191},
      ISBN = {90-6196-208-0},
   MRCLASS = {05C99 (05C65 06B99)},
  MRNUMBER = {605838},
MRREVIEWER = {M.\ E.\ Watkins},
}

@book {Arrow1951,
    AUTHOR = {Arrow, Kenneth J.},
     TITLE = {Social {C}hoice and {I}ndividual {V}alues},
    SERIES = {Cowles Commission Monograph},
    VOLUME = {No. 12},
 PUBLISHER = {John Wiley \& Sons, Inc., New York; Chapman \& Hall, Ltd.,
              London},
      YEAR = {1951},
     PAGES = {xi+99},
   MRCLASS = {90.0X},
  MRNUMBER = {39976},
MRREVIEWER = {D.\ Gale},
}

@article{Sen1966,
 ISSN = {00129682, 14680262},
 URL = {http://www.jstor.org/stable/1909947},
 author = {Amartya K. Sen},
 journal = {Econometrica},
 number = {2},
 pages = {491--499},
 publisher = {[Wiley, Econometric Society]},
 title = {A Possibility Theorem on Majority Decisions},
 urldate = {2026-05-22},
 volume = {34},
 year = {1966}
}

@book{Condorcet1785,
  title={Essai sur l'application de l'analyse {\`a} la probabilit{\'e} des d{\'e}cisions rendues {\`a} la pluralit{\'e} des voix},
  author={de Caritat, Marie Jean Antoine Nicolas and De Condorcet, Marquis},
  year={1785}
}

@article{Black1948,
  title={On the rationale of group decision-making},
  author={Black, Duncan},
  journal={Journal of political economy},
  volume={56},
  number={1},
  pages={23--34},
  year={1948},
  publisher={The University of Chicago Press}
}

@article{Ward1961,
  title={Majority rule and allocation},
  author={Ward, Benjamin},
  journal={Journal of Conflict Resolution},
  volume={5},
  number={4},
  pages={379--389},
  year={1961},
  publisher={Sage Publications Sage CA: Thousand Oaks, CA}
}

@article {GalambosReiner2008,
    AUTHOR = {Galambos, \'Ad\'am and Reiner, Victor},
     TITLE = {Acyclic sets of linear orders via the {B}ruhat orders},
   JOURNAL = {Soc. Choice Welf.},
  FJOURNAL = {Social Choice and Welfare},
    VOLUME = {30},
      YEAR = {2008},
    NUMBER = {2},
     PAGES = {245--264},
      ISSN = {0176-1714,1432-217X},
   MRCLASS = {91B12 (91B08)},
  MRNUMBER = {2372925},
MRREVIEWER = {Vicki\ Knoblauch},
       DOI = {10.1007/s00355-007-0228-1},
       URL = {https://doi.org/10.1007/s00355-007-0228-1},
}

@article{karpovConstructingLargePeakpit2023,
  title = {Constructing Large Peak-Pit {{Condorcet}} Domains},
  author = {Karpov, Alexander and Slinko, Arkadii},
  year = 2023,
  month = jan,
  journal = {Theory and Decision},
  volume = {94},
  number = {1},
  pages = {97--120},
  issn = {0040-5833, 1573-7187},
  doi = {10.1007/s11238-022-09878-9},
  urldate = {2026-02-20},
  langid = {english}
}
\end{document}